\documentclass[11pt,a4paper]{article}
\usepackage[a4paper,margin=1in,footskip=0.25in]{geometry}

\usepackage{soul}
\usepackage{preamble}

\newcounter{theorempart}[theorem]

\numberwithin{equation}{section} 
\usepackage{cleveref}
\Crefname{theorempart}{Theorem}{Theorems}
\newenvironment{cleverproof}[1]
  {\begin{proof}[Proof of \Cref{#1}]}
  {\end{proof}}
\usetikzlibrary{patterns}
\usepackage{booktabs}

\newcommand{\E}{\mathbb{E}}
\newcommand{\N}{\mathbb{N}}

\newcommand{\Prob}{\mathbb{P}}

\allowdisplaybreaks

\title{Bayesian Optimal Stopping with Maximum Value Knowledge}
\author{Pieter Kleer and Daan Noordenbos\\\
Tilburg University\\
Department of Econometrics and Operations Research\\
\texttt{\{p.s.kleer,d.noordenbos\}@tilburguniversity.edu}}
\date{\today}

\begin{document}
\maketitle
\thispagestyle{empty}
\begin{abstract}
\noindent We consider an optimal stopping problem with $n$ correlated offers where the goal is to design a (randomised) stopping strategy that maximises the expected value of the offer at which we stop. Instead of assuming to know the complete correlation structure, which is unrealistic in practice, we only assume to have knowledge of the distribution of the maximum value $X_{\max}$ of the sequence, and want to analyse the worst-case correlation structure whose maximum follows this distribution. This can be seen as a trade-off between the setting in which no distributional information is known, and the Bayesian setting in which the (possibly correlated) distributions of all the individual offers are known.

As our first main result we show that a deterministic threshold strategy using the monopoly price of the distribution of the maximum value is asymptotically optimal assuming that the expectation of the maximum value grows sublinearly in $n$. In our second main result, we further tighten this bound by deriving a tight quadratic convergence guarantee for sufficiently smooth distributions of the maximum value.

Our results also give rise to a more fine-grained picture regarding prophet inequalities with correlated values, for which distribution-free bounds only yield a performance guarantee of the order $1/n$.
\end{abstract}

\newpage
\pagenumbering{arabic}
\setcounter{page}{1}   

\section{Introduction}
We study an optimal stopping problem where a seller has a single item for sale and sequentially receives $n$ offers $(v_1,\dots,v_n)$ from potential buyers.
When an offer is received, the seller must decide whether to accept the offer, and thereby passing up all future offers, or reject the offer and continue to the next one.
The goal is to maximise the (expected) value of the accepted offer, where the expectation is with respect to the randomness of the chosen stopping rule and distributional information regarding the offer sequence (if any).

When no distributional information about the sequence is known, but only its length $n$, it is folklore that the best stopping rule is to choose one of the $n$ points in the sequence uniformly at random as stopping point. This rule yields, in expectation over the randomness of the stopping rule, at least $1/n$ times the (actual) maximum offer value. When the offer values are drawn from a (possibly correlated) known joint distribution, Hill and Kertz \cite{hill1983stop} show that the best one can hope for is a stopping rule that yields an expected stopping value of the order $\mathbb{E}[X_{\max}]/n$, with $X_{\max}$ the maximum value of the sequence, if no further assumptions are made about the joint distribution; this is known as a correlated \textit{prophet inequality} result. There is a plethora of work on prophet inequalities, especially when the values of the offers are not correlated, but independent. We discuss the prophet inequality literature more extensively in Section \ref{sec:rel_work}.
Most relevant for us are recent works in which more structured correlations models have been explored in the context of prophet inequalities.
Immorlica et al.~\cite{immorlica2023prophet} consider a linear correlation model and show that the competitive ratio degrades gradually with increasing correlation.
Livanos et al.~\cite{livanos2024improved} obtain similar results for offers with graphical dependence.

These results seem to suggest that it is correlation that degrades the competitive ratio.
We provide a more nuanced perspective, namely that correlation only degrades the performance when the maximum offer behaves pathologically.

We therefore consider a new model, where instead of making assumptions about the correlation structure, we assume distributional knowledge about the maximum offer value. 
This is arguably a more realistic assumption than knowing the complete correlation structure, since correlations are difficult to estimate from data. This difficulty arises, because specifying correlations often requires many parameters and the correlation structure may change over time. Furthermore, only partial information, like the highest offer, may be available.
In contrast, the distribution of the maximum offer is often much easier to estimate from historical data or could more easily be approximated by an expert, making our model more robust and practical to use.

The goal is now to quantify the worst-case correlation structure whose maximum offer value follows the chosen distribution, against the best (randomised) stopping rule. That is, we first choose a stopping rule after which our adversary \textit{nature} may choose the worst-case correlation structure against this rule. More formal, we want to solve the sup-inf problem
$$
\sup_{\tau\in\mathcal{T}} \inf_{\boldsymbol{X}\in\mathcal{I}(F, n)} \mathbb{E}[\boldsymbol{X}_{\tau}],
$$
where $\mathcal{T}$ is the set of all randomised stopping rules, and $\mathcal{I}(F,n)$ the set of all joint distributions over sequences of $n$ values $(v_1,\dots,v_n)$ where the maximum of the sequence follows $F$; such a distribution is denoted by $\mathbf{X}$. We emphasise that the distribution $F$ is allowed to depend on $n$. Finally, $\mathbf{X}_{\tau}$ is the random offer value at which we stop. Formal definitions are given in Section \ref{sec:model}. For clarity, we will give an example of what instances in $\mathcal{I}(F, n)$ might look like.

\begin{example}
Suppose the distribution $F = F_{\max}$ is given by $\Prob(X_{\max} = 1) = 1/2$ and $\Prob(X_{\max} = 2) = 1/2$, and that $n = 4$. Then example elements (i.e., correlation structures) in $\mathcal{I}(F, n)$ are given in Table \ref{tab:instances}.
\begin{table}[h!]
    \centering
    \begin{minipage}[t]{0.48\textwidth}
    \centering
    \begin{tabular}{lcc}
        \toprule
        \textbf{Instance} & \( v_{\max} \) & \textbf{Probability} \\
        \midrule
        (0,1,1,0) & 1 & \( \frac{1}{2} \) \\[1ex]
        (2,0,1,1) & 2 & \( \frac{1}{4} \) \\[1ex]
        (0,1,\( \frac{3}{2} \),2) & 2 & \( \frac{1}{4} \) \\
        \bottomrule
    \end{tabular}
    \end{minipage}
    \hfill
    \begin{minipage}[t]{0.48\textwidth}
    \centering
    \begin{tabular}{lcc}
        \toprule
        \textbf{Instance} & \( v_{\max} \) & \textbf{Probability} \\
        \midrule
        (0,1,0,1) & 1 & \( \frac{1}{3} \) \\[1ex]
        (1,0,0,$\frac{1}{2}$) & 1 & \( \frac{1}{6} \) \\[1ex]
        (2,0,0,2) & 2 & \( \frac{1}{2} \) \\
        \bottomrule
    \end{tabular}
    \end{minipage}
    \caption{Two examples of elements $\mathbf{X} \in \mathcal{I}(F,n)$ with $\Prob(X_{\max} = 1) = \frac{1}{2}$ and $\Prob(X_{\max} = 2) = \frac{1}{2}$.}
    \label{tab:instances}
\end{table}

\end{example}
It turns out that the above sup-inf problem is equivalent to its inf-sup variant (by an application of the Von Neumann-Fan theorem; see Theorem \ref{thm:minimax-equality}), i.e., 
$
\sup_{\tau\in\mathcal{T}}\inf_{\boldsymbol{X}\in\mathcal{I}(F, n)}\mathbb{E}[\boldsymbol{X}_{\tau}]=\inf_{\boldsymbol{X}\in\mathcal{I}(F, n)}\sup_{\tau\in\mathcal{T}}\mathbb{E}[\boldsymbol{X}_{\tau}].
$
Because, for fixed $n$, the quantity $\mathbb{E}[X_{\max}]$ is constant, if we assume the expectation of $X_{\max}$ is finite, the inf-sup problem is equivalent to
$$
\inf_{\boldsymbol{X}\in\mathcal{I}(F_{\max}, n)}\sup_{\tau\in\mathcal{T}} \; \frac{\mathbb{E}[\boldsymbol{X}_{\tau}]}{\mathbb{E}[X_{\max}]}.
$$
This is called a \textit{prophet inequality} problem as it compares the expected payoff of a stopping rule against that of a ``prophet" who knows the complete value sequence and can always stop at the maximum value.

\subsection{Our contributions}
The main goal of our work is to understand the sup-inf problem above, which is formally stated in \eqref{eq:definition-v-star} in the next section. As a preliminary result, we show in Theorem \ref{thm:monopoly_worstcase} that the optimal deterministic stopping rule is a static threshold whose value is the \textit{monopoly price} 
$$
p^* = \text{argmax}_{ p \geq 0} \; p \cdot\mathbb{P}(X_{\max}\ge p) =:  \text{argmax}_{ p \geq 0} \; \Pi(p).
$$
That is, we select the first offer whose value exceeds $p^*$. Under this stopping rule, the expected value of the sup-inf problem is equal to 
$
\Pi^*(F_{\max}) :=  p^* \cdot\mathbb{P}(X_{\max}\ge p^*), 
$
i.e.,
$$
\Pi^*(F_{\max}) = \sup_{\tau\in\mathcal{D}} \inf_{\boldsymbol{X}\in\mathcal{I}(F_{\max}, n)} \mathbb{E}[\boldsymbol{X}_{\tau}],
$$
where $\mathcal{D}$ is the set of all deterministic stopping rules.

Motivated by Theorem \ref{thm:monopoly_worstcase}, we want to understand whether or not we can obtain a higher expected value by using a randomised stopping rule from $\mathcal{T}$. {Determining the exact randomized stopping rule that maximizes the expected offer value at which we stop seems rather difficult and cumbersome. Therefore, we instead aim for bounds on the performance of randomized algorithms that capture their potential.} In Section \ref{sec:first_bounds} we show in Theorem \ref{thm:non-optimality-determinstic} that, under a mild differentiability assumption on the cumulative distribution function of $F_{\max}$, there exists a randomised stopping rule that outperforms the best deterministic strategy. In contrast, we show in Theorem \ref{thm:randomised-bound-continuous-xmax} that for \textit{any} $F_{\max}$, 
$$
\left( \Pi^*(F_{\max})  \leq \; \right) \; \; \sup_{\tau\in\mathcal{T}} \inf_{\boldsymbol{X}\in\mathcal{I}(F_{\max}, n)} \mathbb{E}[\boldsymbol{X}_{\tau}] \leq \Pi^*(F_{\max}) + \frac{\mathbb{E}[X_{\max}]}{n}.
$$
This means that as long as the expectation of $X_{\max}$ grows sublinearly, the advantage of randomization disappears as $n$ grows large. We regard this to be a reasonable assumption on $X_{\max}$.\footnote{A notable exception to this assumption could be the case in which the market of buyers represents a ``stock market bubble" in which the maximum value for an item increases steeply as more buyers become interested in it.}

In Section \ref{sec:tight_bounds} we tighten the results of Section \ref{sec:first_bounds} for distributions $F_{\max}$ that are sufficiently smooth. In particular, we show in Theorem \ref{thm:special-cases} that if $F_{\max}$ is three times differentiable around $p^*$ and independent of $n$, then 
$$
\left( \Pi^*(F_{\max}) \leq \; \right) \; \;  \sup_{\tau\in\mathcal{T}} \inf_{\boldsymbol{X}\in\mathcal{I}(F_{\max}, n)} \mathbb{E}[\boldsymbol{X}_{\tau}] \leq \Pi^*(F_{\max})+\Theta\left(n^{-2}\right).
$$
Well-known parametric distributions such as the exponential, (half)-normal, and gamma distribution satisfy the conditions of Theorem \ref{thm:special-cases}.

Finally, in Section \ref{sec:prophets} we investigate natural conditions on $F_{\max}$ under which the prophet inequality inf-sup problem $\inf_{\boldsymbol{X}\in\mathcal{I}(F, n)}\sup_{\tau\in\mathcal{T}} \; \mathbb{E}[\boldsymbol{X}_{\tau}]/\mathbb{E}[X_{\max}]$ yields a lower bound better than of the order $1/n$, when $F_{\max}$ is itself allowed to depend on $n$. 
\noindent 
\paragraph{Mechanism design implications.} In all our results, the static monopoly threshold $p^*$ arises (roughly speaking) as the asymptotically optimal stopping rule. From a mechanism design point of view this also has desirable implications. Consider the \textit{online mechanism design} setting where we interpret the threshold as a price for the item, and the offer from a potential \textit{strategic} buyer as the price they are willing to pay for it. When using a fixed threshold as selling price, an example of a so-called \textit{posted price mechanism}, buyers have no incentive to misreport their (true) value for the item as this can never increase their utility (which is the maximum of their true value minus the paid price if they get the item, and zero otherwise). Such posted price mechanisms have close connections with prophet inequalities see, e.g., \cite{hajiaghayi2007automated,chawla2010multi,correa2019pricing}.

{
\paragraph{Our techniques.} To establish our lower bound results, it turns out that we can employ stopping rules with uniformly random thresholds, i.e., we randomly choose a threshold from a given interval and accept the first offer exceeding that threshold. Such policies are analytically tractable and allow us to characterise the worst-case (increasing) value sequences against such a stopping policy. To make the computation of the expected stopping value under these worst-cases sequences tractable, we use (in Section \ref{sec:tight_bounds}) a Taylor approximation of the maximum value distribution, which is possible because the random threshold is chosen from a small interval surrounding $p^*$.}

{
For the upper bounds, we construct worst-case instances by partitioning the support of $X_{\max}$ into $n$ intervals.
We then bound the expected payoff by factoring out the stopping rule, yielding a bound in terms of the partition only. Determining the optimal partition comes, roughly speaking, down to an application of the \textit{intermediate value theorem} in Section \ref{sec:first_bounds}. In Section \ref{sec:tight_bounds} we need a highly non-trivial recursive construction to determine the optimal partition. This construction (to be found in Appendix \ref{proofs:technical-lemmas}) might be of independent interest.
}

\subsection{Related work}
\label{sec:rel_work}
{The sup-inf stopping problem above can be interpreted as a \textit{robust} stopping problem, in the sense that we first choose a stopping rule, after which our adversary \textit{nature} chooses a worst-case correlation structure adhering to our choice of maximum value distribution. Boshuizen and Hill \cite{boshuizen1992moment} consider a setting in which the offer values are independently distributed. Each offer comes from an unknown distribution of which only the mean and variance are known (the set of distributions adhering to the given mean and variance is called the \textit{mean-variance ambiguity set}). The goal is to design a stopping rule that maximizes the seller's expected payoff assuming nature chooses a worst-case distribution adhering to the chosen means and variances. Kleer and Van Leeuwaarden \cite{kleer2022optimalstoppingtheorydistributionally} generalize this setting to arbitrary ambiguity sets for the independent offers. (These problems can also be interpreted as robust Markov decision processes in the sense of, e.g., \cite{iyengar2005robust}.) In both works \cite{boshuizen1992moment,kleer2022optimalstoppingtheorydistributionally}, the authors explicitly determine the optimal stopping rule, which has for every offer a threshold value for accepting it. These thresholds can be computed by means of a backward recursion. This results in a deterministic optimal stopping rule, as opposed to our worst-case correlated setting, where in most cases randomized stopping rules are optimal.}

The (in terms of value) equivalent inf-sup problem is closely related to the literature on prophet inequalities, where the complete correlation structure of the values is assumed to be known.
Hill and Kertz \cite{hill1983stop} show that without any further assumptions on the (known) correlation structure $\sup_{\tau\in\mathcal{T}} \; \mathbb{E}[\boldsymbol{X}_{\tau}]/\mathbb{E}[X_{\max}] = \Theta(1/n)$.\footnote{We write $f(x) = O(g(x))$ if $f(x) \leq c g(x)$ for some constant $c$ and all $x$, and $f(x) = \Omega(g(x))$ if $f(x) \geq c g(x)$ for some constant $c$ and all $x$. We write $f(x) = \Theta(g(x))$ if $c_1g(x) \leq f(x) \leq c_2g(x)$ for constants $c_1, c_2$ and $x \geq 0$.}
At the other extreme, if the offer values are completely independent, Krengel, Sucheston and Garling \cite{krengel1978semiamarts} show that there is a $1/2$-competitive policy, i.e., $\sup_{\tau\in\mathcal{T}} \; \mathbb{E}[\boldsymbol{X}_{\tau}]/\mathbb{E}[X_{\max}] \geq 1/2$.
Samuel-Cahn \cite{samuel1984comparison} showed that this guarantee can be achieved with a static threshold policy using as threshold the median of the distribution of the maximum of the independent values. Kleinberg and Weinberg \cite{kleinberg2012matroid} showed that this guarantee can also be achieved with as threshold half the expectation of the maximum value. Our optimal policy is similar to these optimal policies, in the sense that both are static threshold policies based on the maximum value distribution.
The advantage of the Kleinberg-Weinberg threshold for independent values is that it can be generalized to the setting where multiple offers can be accepted subject to a combinatorial matroid constraint; related to such multi-selection problems is also the setting of \textit{online combinatorial auctions}, see, e.g., the recent work by Cristi and Correa \cite{correa2023constant}, and references therein.

The competitive guarantee of $1/2$ can be improved under the additional assumption that the values are all drawn from the same distribution (the i.i.d. setting). For the case where the common distribution is known, Correa et al. \cite{correa2017posted} have solved this problem by establishing a tight $0.745$ competitive ratio.
This result is generalised in \cite{correa2022prophet}, by considering the setting of an unknown common distribution from which only a limited number of samples are available. For further literature with only sample access to the distribution, we refer to the recent work of Cristi and Ziliotto \cite{cristi2024prophet}, and references therein. If the base distribution does not depend on $n$, Kennedy and Kertz \cite{kennedy1991asymptotic}
show that $0.745$ can be beaten, and Correa et al. \cite{correa2021optimal} show that there are also better threshold policies.
To sensibly define this problem they make extreme value theory assumptions. 
For our problem such assumptions result in the monopoly price threshold being asymptotically optimal, since it implies smoothness and sublinear growth of $\E[X_{\max}]$ (that we require for our results in Section \ref{sec:first_bounds} and \ref{sec:tight_bounds}). 

There are also works that interpolate between these two extremes of no dependence restrictions and independence.
Immorlica et al. \cite{immorlica2023prophet}
consider a linear correlations model, where the $n$ offers are random variables $\mathbf{X} = (X_1,\dots,X_n)$ which are a linear combination of $m$ independent non-negative random variables $\boldsymbol{Y} = (Y_1,\dots,Y_m)$, via $\boldsymbol{X}=A\boldsymbol{Y}$ with $A$ a non-negative matrix.
Crucial to their work are
the row-sparsity $s_{\text{row}}$, the maximum number of features $Y_j$ that influence any offer $X_i$, and the column-sparsity $s_{\text{col}}$, the maximum amount offers $X_i$ influenced by any feature $Y_j$.
Their main result is that even with full knowledge of $A$ and the distribution of $\boldsymbol{Y}$, no online algorithm can recover more than a fraction $O(1/\min\{s_{\text{row}},s_{\text{col}})\})$
of the maximum. This result therefore interpolates between the two extremes of independent and fully correlated values.
Livanos et al. \cite{livanos2024improved} obtain similar interpolation results for models 
graphical dependencies 
Caragiannis et al. \cite{caragiannis2021relaxing} consider prophet inequalities with pairwise independent offer values, thereby also relaxing the independence assumption. The above exposition of prophet inequality literature is not meant to be exhaustive. We also refer to reader to, e.g., the survey of Lucier \cite{lucier2017economic} for further prophet-related settings. 

Finally, in the context of revenue maximization in auction design there is also a line of work \cite{gravin2018,bei2019correlation,babaioff2020escaping} building on the \textit{correlation-robust} framework of Carroll \cite{carroll2017robustness} in which the seller has marginal distributional information about the items, but not about their correlation structure. Such settings might also be interesting to study in the context of online stopping.

\section{Model}
\label{sec:model}
We consider the problem where $n$ unknown values $v_1,\dots,v_n$ are revealed sequentially. 
When a value is revealed an irrevocable decision has to be made whether to accept it, disregarding all future values, or continue to the next value.

An algorithm for deciding at which value to stop is called a stopping rule (or policy). A randomised stopping rule $\tau$ is a collection of functions $(\tau_i)_{i = 1,\dots,n}$, where $\tau_i$ is a Bernoulli random variable that takes as inputs $v_1,\dots,v_i$. Based on this input, the function $\tau_i$ selects value $v_i$ with probability $r_i(v_1,\dots,v_i)$. To be precise, we have
\begin{align*}
    \tau_i(v_1,\dots,v_i) =
    \begin{cases}
        1, & \text{ w.p. } r_i(v_1,\dots,v_i),\\
        0, & \text{ w.p. } 1 - r_i(v_1,\dots,v_i),
    \end{cases}
\end{align*}
for $i = 1,\dots,n$. This means that the probability offer $i$
is selected is given by
\begin{align*}
\Prob(\tau = i  \ | \   v_1,\dots,v_i) = r_i(v_1,\dots,v_i)\prod_{j=1}^{i-1}(1 - r_j(v_1,\dots,v_j)).
\end{align*}
A stopping rule is called deterministic if $r_i(v_1,\dots,v_i)\in\{0,1\}$ for all $(v_1,\dots,v_i)$ and $i$, so for deterministic stopping rules $\Prob(\tau = i  |  v_1,\dots,v_i)\in\{0,1\}$. Randomised stopping rules are therefore a mixture, that is a convex combination, of deterministic stopping rules.
We write
$$
\boldsymbol{v}_{\tau} :=  \sum_{i=1}^n v_i\cdot\Prob(\tau = i |  v_1,\dots,v_i)
$$
for the expected value of the offer that the stopping rule $\tau$ selects on the fixed tuple $\boldsymbol{v}$. \bigskip

\noindent \paragraph{Maximum value knowledge.} We consider an optimal stopping problem with distributional knowledge about the maximum value of the sequence $(v_1,\dots,v_n)$.
Specifically, we assume that the tuple of values $(v_1,\dots,v_n)$ is the realization of a nonnegative random vector $\mathbf{X} = (X_1,\dots,X_n)$ that follows a joint distribution $F_\mathbf{X}$ which we assume is given as a (joint) cumulative distribution function (CDF). We do not assume to know this joint distribution, but, instead, only the distribution $F_{\max}$ of $\max_{1 \leq i \leq n} X_i$. For a given $F_{\max}$, we denote with $\mathcal{I}(F_{\max},n)$ the set of all joint distributions over $n$ values whose maximum value follows $F_{\max}$, i.e.,
\begin{align}
    \mathcal{I}(F_{\max},n)=\{F_{\boldsymbol{X}}:\mathbb{R}^n_+\to[0,1]\;|\;F_{\boldsymbol{X}}\text{ a CDF and }
\max_{1\le i\le n}\{X_i\}\sim F_{\max}\}.\label{eq:instance-ambiguity-set}
\end{align} 
With a slight abuse of notation, we sometimes write $\boldsymbol{Y}\in\mathcal{I}(F,n)$ to indicate that the joint distribution $F_{\mathbf{Y}}$ of $\mathbf{Y}$ is in $\mathcal{I}(F,n)$.

Fundamental in the optimal stopping problem with maximum value knowledge is the revenue function 
$$
    \Pi(p)=p\cdot\mathbb{P}(X_{\max}\ge p),
$$
where $X_{\max}$ is a random variable with distribution $F_{\max}$. This notation will be used throughout. With the revenue function we can define the monopoly price of $F_{\max}$ as
\begin{align}
    p^*=\underset{p\ge 0}{\mathrm{argmax}}\{\Pi(p)\}\text{ with }  \Pi^* = \Pi(p^*).
    \label{eq:monopoly_price}
\end{align}
The monopoly price will be ubiquitous in our analysis.

For a given distribution $F_{\max}$ and set of admissible stopping rules $\mathcal{T}$, which without further specification is the set of randomised stopping rules, we want to find the stopping rule $\tau \in \mathcal{T}$ that solves
\begin{align}
    V^*(F_{\max}, n):=\sup_{\tau\in\mathcal{T}}\inf_{F_{\boldsymbol{X}}\in\mathcal{I}(F_{\max}, n)}\int_{R^n_+}\boldsymbol{v}_{\tau}\mathrm{d}F_{\boldsymbol{X}}(\boldsymbol{v})=\sup_{\tau\in\mathcal{T}}\inf_{\boldsymbol{X}\in\mathcal{I}(F_{\max}, n)}\mathbb{E}[\boldsymbol{X}_{\tau}].\label{eq:definition-v-star}
\end{align}
That is, we first choose a stopping rule $\tau$, after which our adversary \textit{nature} chooses a worst-case joint distribution from $\mathcal{I}(F_{\max},n)$ whose maximum value follows $F_{\max}$. We then consider the expected performance of the stopping rule $\tau$ under a tuple of values drawn from the worst-case distribution/instance, and the possible randomness of $\tau$. 
\subsection{Minimax equality}
The maximin problem described in \eqref{eq:definition-v-star} yields the same value as the minimax problem 
\begin{align}
    \inf_{\boldsymbol{X}\in\mathcal{I}(F, n)}\sup_{\tau\in\mathcal{T}}\mathbb{E}[\boldsymbol{X}_{\tau}] \label{eq:minimax-v-star}
\end{align}
This follows from the Von Neumann-Fan minimax theorem, as we illustrate in Theorem \ref{thm:minimax-equality}. The proof is deferred to Appendix \ref{proof:minimax-equality}.

\begin{theorem}
With $\mathcal{T}$ the set of all randomised stopping rules, and $\mathcal{I}(F,n)$ the set of all joint distributions over $n$ values whose maximum follows $F$, we have
\label{thm:minimax-equality}
\begin{align}   \sup_{\tau\in\mathcal{T}}\inf_{\boldsymbol{X}\in\mathcal{I}(F, n)}\mathbb{E}[\boldsymbol{X}_{\tau}]=\inf_{\boldsymbol{X}\in\mathcal{I}(F, n)}\sup_{\tau\in\mathcal{T}}\mathbb{E}[\boldsymbol{X}_{\tau}].
\end{align}
\end{theorem}

The intuition underlying Theorem \ref{thm:minimax-equality} is that we do not impose any further restrictions on joint distributions $X \in \mathcal{I}(F,n)$, i.e., arbitrary correlations between values are allowed. This results in the set $\mathcal{I}(F,n)$ being convex. 
This contrasts with settings that restrict correlation. 
Let $\tilde{\mathcal{I}}(F,n)\subset\mathcal{I}(F,n)$ satisfy the additional condition that the joint distributions are a product distribution, meaning that the $X_i$ are assumed to be independent, then $\tilde{\mathcal{I}}(F,n)$ is (in general) not convex.
\begin{example}
    Let $F\sim\mathrm{Bernoulli}(\frac{1}{2})$ and $X_i\sim\mathrm{Bernoulli}(\frac{1}{2})$, then two independent instances in $\tilde{\mathcal{I}}(F,2)$ are $(0,X_1)$ and $(X_2,0)$.
    Let the instance $(Y_1,Y_2)$ be a mixture which is the first instance with probability $\frac{1}{2}$ and the second instance with probability $\frac{1}{2}$, then
    $$
    (Y_1,Y_2)=
    \begin{cases}
        (0,0),&\text{w.p. }\frac{1}{2},\\
        (0,1),&\text{w.p. }\frac{1}{4},\\
        (1,0),&\text{w.p. }\frac{1}{4}.
    \end{cases}
    $$
    We have that $\max\{Y_1,Y_2\}\sim F$, but $\mathbb{P}(Y_1=0)=\frac{3}{4}$ and $\mathbb{P}(Y_1=0|Y_2=0)=\frac{2}{3}$, so $Y_1$ and $Y_2$ are not independent. Therefore $(Y_1,Y_2)\notin\tilde{\mathcal{I}}(F,2)$. 
\end{example}

\subsection{Deterministic stopping rules}
Here we consider the setting where the set of admissible stopping rules is the set of deterministic stopping rules.
A natural subset of deterministic stopping rules are (deterministic) threshold policies, which accept the first value exceeding a fixed threshold.
\begin{definition}[Deterministic threshold policy]
    \label{def:threshold-policy}
    A threshold policy $\tau$ with deterministic threshold $t$ has the following stopping rule:
    \begin{align*}
        \tau_i(v_1,\dots,v_i)&=\indicator{v_i\ge t}\text{ for }i=1,\dots,n.
    \end{align*}
\end{definition}
We establish that a threshold policy is the best deterministic policy in Theorem \ref{thm:monopoly_worstcase}.

\begin{theorem}
For $n\ge 2$ the optimal deterministic stopping rule has a payoff of $\Pi^*(F_{\max})$ and is a threshold policy with the monopoly price as threshold.
\label{thm:monopoly_worstcase}
\end{theorem}
\begin{proof}
    Let $\mathcal{I}(p)$ denote the following family of instances
$$
(v_1,\dots,v_n) = 
\begin{cases}
    (p,v_{\max},0,\dots,0), & \text{ if } v_{\max} \geq p,\\
(v_{\max},0,0,\dots,0),& \text{ if } v_{\max} < p.
\end{cases}
$$
Moreover, let $\Pi^*=\Pi^*(F_{\max})$.
The payoff of the static monopoly price threshold is at least $p^*$ if $X_{\max}\ge p^*$, so the expected payoff is at least $p^*\cdot\mathbb{P}(X_{\max}\ge p^*)=\Pi^*$. Nature can make sure the expected payoff is never more than $\Pi^*$ by choosing the instance $\mathcal{I}(p^*)$.
We will now show that no other deterministic stopping rule $\tau = (\tau_i)$ can do better. Let $w=\inf\{v_1:r_1(v_1)=1\}$, where $r_1$ is given by $\tau_1$. 

If $w=\infty$, meaning that $v_1$ is never selected by $\tau$, then the instance $(v_{\max},0,\dots,0)$ guarantees a payoff of zero. If $w$ is finite and the infimum is attained, then the instance $\mathcal{I}(w)$ yields a payoff of $w\cdot\mathbb{P}(X_{\max}\ge w)\le \Pi^*$. 

Lastly, if $w$ is finite and the infimum is not attained we can use essentially the same argument as when the infimum is attained by taking a value slightly higher than $w$. 
Formally, choose $\delta>0$ arbitrarily and choose $x=x(\delta)$ with $r_1(x)=1$, such that $\mathbb{P}(\{v_1:v_1\le x\text{ and }r_1(x)=1\})\le\delta$. Now for the instances $\mathcal{I}(x)$ the payoff is at most $x\cdot\mathbb{P}(X_{\max}\ge x)+x\cdot\delta\le\Pi^*+x\cdot \delta$. Letting $\delta\to 0$ yields the desired result, as $x(\delta)$ can be chosen as a decreasing function of $\delta$. 
\end{proof}

\section{Bounds for randomised stopping rules}
\label{sec:first_bounds}
Deterministic stopping rules are not always optimal.
For example, consider the case where $n=2$ and $X_{\max}$ has an exponential distribution with mean $1$, i.e., its CDF is $F(x) = 1 - e^{-x}$ for $x \geq 0$. Then the worst-case expected value obtained by setting the monopoly price as a (deterministic) threshold is $\Pi^*(F_{\max})=\frac{1}{e}$, but picking 
one of the two values uniformly at random has an expected payoff of $\frac{1}{2}$, so here randomization clearly helps.

In this section, we provide a natural condition under which a randomised stopping rule strictly outperforms any deterministic stopping rule. Additionally, we establish a universal upper bound on $V^*(F_{\max},n)$ that shows this improvement is typically only marginal. These findings are Theorems \ref{thm:non-optimality-determinstic} and \ref{thm:randomised-bound-continuous-xmax}, respectively. For ease of readability, we repeat the definitions of $\Pi^*$ and $V^*$ in the theorem statements.

\begin{theorem}
    Let $n \in \N$ and let $F_{\max}$ be a given probability distribution (possibly depending on $n$). Define
    $$
    \Pi^*(F_{\max}) = \max_{ p \geq 0} p \cdot\mathbb{P}(X_{\max}\ge p) = p^* \cdot\mathbb{P}(X_{\max}\ge p^*)
    $$
    as the worst-case expected value of threshold algorithm using monopoly price $p^*$, and 
    $$
    V^*(F_{\max}, n) = \sup_{\tau\in\mathcal{T}}\inf_{\boldsymbol{X}\in\mathcal{I}(F_{\max}, n)}\mathbb{E}[\boldsymbol{X}_{\tau}]
    $$
    the worst-case expected payoff of the optimal randomised rule for sequences of $n$ (correlated) values whose maximum follows $F_{\max}$.
    \begin{itemize}
        \item [(a)] \refstepcounter{theorempart}\label{thm:non-optimality-determinstic}
        If $F_{\max}$ is differentiable at $p^*$ then
        $V^*(F_{\max}, n) > \Pi^*(F_{\max})$.
        \item [(b)] \refstepcounter{theorempart}\label{thm:randomised-bound-continuous-xmax} For any $F_{\max}$, the following upper bound holds: 
        $$
        V^*(F_{\max}, n) \le \Pi^*(F_{\max}) + \frac{\mathbb{E}[X_{\max}]}{n}.
        $$
    \end{itemize}    
    \label{thm:first_bounds}
\end{theorem}

Theorem \ref{thm:non-optimality-determinstic} is not true without any condition on $F_{\max}$. This is illustrated in the following lemma, where we show that that for discrete distributions with at most $n$ support points, the monopoly price threshold algorithm remains optimal, even among all randomised stopping rules.

\begin{lemma}
    If $X_{\max}$ has at most $n$ support points, then $V^*(F_{\max},n)=\Pi^*(F_{\max})$ with $V^*(F_{\max},n)$ and $\Pi^*(F_{\max})$ as in Theorem \ref{thm:first_bounds}.
\end{lemma}
\begin{proof}
    Denote the support points of $X_{\max}$ by $x_1\le \dots\le x_n$ and suppose nature chooses for $i = 1,\dots,n$ the tuple $(x_1,\dots,x_i,0,\dots,0)$ with probability 
    $\mathbb{P}(X_{\max} = x_i)$. Then clearly the maximum of the tuple values follows the distribution 
$F_{\max}$, i.e., the distribution over the sequences is contained in $\mathcal{I}(F_{\max},n)$.
    The expected payoff can then be bounded by
    \begin{align*}
       V^*(F_{\max},n) &=\sum_{j=1}^n\mathbb{P}(X_{\max}=x_j)\sum_{i=1}^jx_i\mathbb{P}(\tau=i|x_1,\dots,x_i)\\
       &=\sum_{i=1}^n\mathbb{P}(\tau=i|x_1,\dots,x_i)\sum_{j=i}^nx_i\mathbb{P}(X_{\max}=x_j)\\
        &=\sum_{i=1}^n\mathbb{P}(\tau=i|x_1,\dots,x_i)\cdot x_i\mathbb{P}(X_{\max}\ge x_i)\\
        &\le\sum_{i=1}^n\mathbb{P}(\tau=i|x_1,\dots,x_i)\cdot\Pi^*(F_{\max})\le \Pi^*(F_{\max}).
    \end{align*}
    In the last step we use the fact that the probabilities of stopping add up to at most $1$.
    Lastly, by Theorem \ref{thm:monopoly_worstcase} $V^*(F_{\max},n)\ge\Pi^*(F_{\max})$, as deterministic stopping rules are a subset of randomised stopping rules, so $V^*(F_{\max},n)=\Pi^*(F_{\max})$
\end{proof}

\subsection{Non-optimality deterministic stopping rules}
In this section we prove Theorem \ref{thm:non-optimality-determinstic}, i.e., we show that, under a natural condition, randomised stopping rules can outperform deterministic ones. We do this by using a \textit{randomised threshold policy} as defined in Definition \ref{def:RTP}. Such a policy chooses a threshold at random according to a given distribution, and then selects the first value that exceeds the threshold.

Structural properties of worst-case instances under random threshold policies are established in Lemmas \ref{lemma:increasing-threshold} and \ref{lemma:support-restriction}.
Using these lemmas, the payoff for a class of simple random threshold policies is then computed in Lemma \ref{lemma:expected-payoff-n-point-policy}, from which Theorem \ref{thm:non-optimality-determinstic} follows directly.

\begin{definition}[Random threshold policy]
    \label{def:RTP}
    A random threshold policy $\tau$ with the random variable $T$, with distribution $G$, as threshold has the following stopping rule:
    \begin{align*}
        \tau_i(v_1,\dots,v_i;T)&=\indicator{v_i\ge T}\text{ for }i=1,\dots,n.
    \end{align*}
\end{definition}

When a random threshold policy is used not all instances need to be considered.
Specifically, we only need to consider instances with increasing values restricted to the support of the random threshold.
In the following two lemmas this is formalised and proven.  
\begin{lemma}
    \label{lemma:increasing-threshold}
    If the stopping rule is a random threshold policy, then for a given realization $v_{\max}$ of $X_{\max}$ and instance $(w_1,\dots,w_n)$, there exists an instance $(v_1,\dots,v_n)$ with $v_1< \dots< v_n=v_{\max}$, such that the payoff under $(v_1,\dots,v_n)$ is not higher than the payoff under $(w_1,\dots,w_n)$.
\end{lemma}
\begin{proof}
    Let 
    $$
    J=\left\{j\in\{1,\dots,n\}:w_j\le \max_{i=1,\dots,j-1} w_i \right\}.
    $$
    Because a threshold policy is used no $w_j$ with $j\in J$ will be accepted, since it would have already been accepted. Let $v_1,\dots,v_n$ be an increasing sequence with $v_n=v_{\max}$ which contains as a subsequence the values $w_j$ for $j\notin J$. 
    For any threshold realisation $t$ of $G$, $v_{\inf\{i:t\le v_i\}}\le w_{\inf\{i\notin J:t\le w_i\}}$, so the payoff under $(v_1,\dots,v_n)$ is not higher. 
\end{proof}

\begin{lemma}
    \label{lemma:support-restriction}
    If the stopping rule is a random threshold policy with random threshold $T$, then for a given realization $v_{\max}$ of $X_{\max}$ and instance $(w_1,\dots,w_n)$, there exists an instance $(v_1,\dots,v_n)$ with $v_i\in\mathrm{supp}(T)$ for $i<n$, such that the payoff under $(v_1,\dots,v_n)$ is not higher than the payoff under $(w_1,\dots,w_n)$. 
\end{lemma}
\begin{proof}
    By Lemma \ref{lemma:increasing-threshold} we can assume without loss of generality that $w_n=v_{\max}$. 
    Let $v_i=\sup\{u\in\mathrm{supp}(T):u\le w_i\}$ for $i<n$ and $v_n=w_n$, that is, $v_i$ is $w_i$ projected downward into $\mathrm{supp}(T)$. By construction, if $w_i$ is accepted, then $v_i$ is also accepted. Moreover $v_i\le w_i$, so the expected payoff under $(v_1,\dots,v_n)$
    is no higher than under the original tuple. 
\end{proof}

Below we determine, using the above structural results, the payoff of a random threshold policy with $n$ support points.

\begin{lemma}
    \label{lemma:expected-payoff-n-point-policy}
    For the random threshold policy with a random threshold that is $t_i$ with probability $p_i$ for $i = 1,\dots,n$, assuming without loss of generality that $t_1<\dots<t_n$, the expected payoff is
    $$
    \min_{\mathclap{\boldsymbol{X}\in\mathcal{I}(F_{\max},n)}} \mathbb{E}[\boldsymbol{X}_{\tau}]=
    \sum_{i=1}^np_it_i \mathbb{P}(X_{\max}\ge t_i)+\gamma\cdot\mathbb{P}\left(X_{\max}\ge t_n+\frac{\gamma}{p_n}\right)+p_n\int\limits_0^{\gamma/p_n}v\mathrm{d}F_{\max}(v+t_n),
    $$
    where $\gamma=\min\{p_1(t_2-t_1),\dots,p_{n-1}(t_n-t_{n-1})\}$.
\end{lemma}
\begin{proof}
    For a given $v_{\max}$, we first determine the optimal $v_i$'s.
    If $v_{\text{max}}<t_1$, then $(v_1,\dots,v_n)=(v_{\text{max}},\dots,0)$ yields a payoff of $0$.
    Suppose that $v_{\max}\in[t_m,t_{m+1})$ for $1\le m < n$. If the realised threshold is less than or equal to $t_m$, then the payoff will at least be the realised threshold. The expected payoff is therefore at least $\sum_{i=1}^mp_it_i$. Setting $(v_1,\dots,v_n)=(t_1,\dots,t_m,v_{\text{max}},\dots,0)$ guarantees this minimal payoff.
    
    The case $t_n\le v_{\text{max}}$ is more involved, but by Lemmas \ref{lemma:increasing-threshold} and \ref{lemma:support-restriction} we only need to consider $(v_1,\dots,v_n)$ increasing and $v_i\in\{t_1,\dots,t_n,v_{\text{max}}\}$. The $n$ options under consideration are therefore, 
    $$(t_2,t_3\dots,t_{n-1}, t_n,v_{\text{max}}),(t_1,t_3\dots,t_{n-1}, t_n,v_{\text{max}}), \dots,\text{ and }(t_1,t_2\dots,t_{n-2},t_{n-1},v_{\text{max}})$$
    If $t_i$ is the missing value and $t_j$ is the realised threshold, then the payoff is $t_j$ if $j\ne i$ and $t_{j+1}=t_j+(t_{j+1}-t_j)$ if $j=i$, using $t_{n+1}=v_{\max}$. The conditional payoff is therefore $\sum_{i=1}^np_it_i+\min\{\gamma,p_n(v_{\max}-t_n)\}$, where $\gamma = \min\{p_1(t_2-t_1),\dots,p_{n-1}(t_n-t_{n-1})\}$.
    The conditional payoffs can be aggregated into the actual expected payoff as follows,
    \begin{align*}
        \mathbb{E}[X_{\tau}]=&
        \sum^{n-1}_{j=1}\left(\sum_{i=1}^jp_it_i\right)\mathbb{P}(t_j\le X_{\max}<t_{j+1})\\
        +&\sum_{i=1}^np_it_i\mathbb{P}(X_{\max}\ge t_n)+\int_{t_n}^\infty\min\{\gamma,p_n(v-t_n)\}\mathrm{d}F_{\max}(v)\\=&\sum_{i=1}^np_it_i \mathbb{P}(X_{\max}\ge t_i)+\gamma\cdot\mathbb{P}\left(X_{\max}\ge t_n+\frac{\gamma}{p_n}\right)+p_n\int_0^{\gamma/p_n}v\mathrm{d}F_{\max}(v+t_n).
    \end{align*}
    This completes the proof.
\end{proof}

With a calculus based argument we can show the existence of an $n$ point random threshold with a payoff larger than $\Pi^*(F_{\max})$, proving Theorem \ref{thm:non-optimality-determinstic}.

\begin{cleverproof}{thm:non-optimality-determinstic}
We define 
$$
\Pi(p) = p \cdot\mathbb{P}(X_{\max} \geq p).
$$
Because $F_{\max}$ is differentiable at $p^*$ we have that $\Pi'(p^*)=0$ and $\mathbb{P}(X_{\max}\ge p^*)>0$. Pick $\epsilon\ge 0$ and let the random threshold be $t_i=p^*-(n-i)\epsilon$ with probability $p_i=\frac{1}{n}$ for $i=1,\dots,n$. The case $\epsilon=0$ corresponds to a deterministic policy using the monopoly price $p^*$ as a (static) threshold. By Lemma \ref{lemma:expected-payoff-n-point-policy} the expected payoff is at least
$$
\tilde{\Pi}(\epsilon)=\frac{1}{n}\sum_{i=0}^{n-1}\Pi(p^*-i\epsilon)+\frac{\epsilon}{n}\mathbb{P}(X_{\max}\ge p^*+\epsilon),
$$
where the third term, which is non-negative, is dropped.
Now $\tilde{\Pi}(0)=\Pi(p^*)$, 
so $\tilde{\Pi}'(0)>0$ implies that
there exists an $\epsilon$ such that $\tilde{\Pi}(\epsilon)>\tilde{\Pi}(0)$, proving the statement in Theorem \ref{thm:non-optimality-determinstic}. 
We complete the proof by showing $\tilde{\Pi}'(0)>0$:
\begin{align*}
        \tilde{\Pi}'(0)&=\lim_{\epsilon\to 0}\frac{1}{\epsilon}(\tilde{\Pi}(\epsilon)-\tilde{\Pi}(0))=
    \lim_{\epsilon\to 0}\frac{1}{\epsilon}\left(\frac{\epsilon}{n}\mathbb{P}(X_{\max}\ge p^*+\epsilon)+\frac{1}{n}\sum_{i=1}^n\left(\Pi(p^*-i\epsilon)-\Pi(p^*)\right)\right)\\
    &=\frac{1}{n}\mathbb{P}(X_{\max}\ge p^*)+ \frac{1}{n}\lim_{\epsilon\to 0}\sum_{i=1}^n(-i)\frac{\left(\Pi(p^*-i\epsilon)-\Pi(p^*)\right)}{-i\epsilon}\\
    &=\frac{1}{n}\mathbb{P}(X_{\max}\ge p^*)+ \frac{1}{n}\sum_{i=1}^n(-i)\underbrace{\lim_{\epsilon\to 0}\frac{\left(\Pi(p^*-i\epsilon)-\Pi(p^*)\right)}{-i\epsilon}}_{=\Pi'(p^*)=0}\\
    &=\frac{1}{n}\mathbb{P}(X_{\max}\ge p^*)>0. \qedhere
\end{align*}
\end{cleverproof}
\subsection{Universal upper bound}
Theorem \ref{thm:randomised-bound-continuous-xmax} is self-contained and relies on a clever instance that reveals little information, making all stopping rules ineffective.

\begin{cleverproof}{thm:randomised-bound-continuous-xmax}
    We prove it here for continuous $F_{\max}$, in Appendix \ref{proof:universal-upperbound-general} the general case is proven. Let $w_0$ and $w_n$ denote the lower and upper bound of the support of $X_{\max}$. Consider a sequence $w_0< w_1\dots< w_n$ that partitions the support into $n$ intervals. Recursively define 
    $$
    b_i= \argmax{x\in[w_{i-1}, w_i]}r_i(b_1,\dots,b_{i-1}, x) \ \text{ for } i=1,\dots,n.
    $$
    Define then for $i = 1,\dots,n$
    $$
v_i(t)=\begin{cases}
    b_i,&\text{if } t> w_i,\\
    t,&\text{if } t\in [w_{i-1},w_i],\\
    0,&\text{if } t< w_{i-1}.
\end{cases}
$$
Here $t$ will later be the realization of the random variable $X_{\max}$ distributed according to $F_{\max}$. Furthermore, roughly speaking, $b_i$ is the value that maximises the probability that under a given stopping rule $\tau$, we stop at the $i$-th value seen if this value lies within the interval $[w_{i-1},w_i]$ and the realised maximum value of $X_{
    \max}$ exceeds $w_i$. In other words, nature will try to make us stop as early as possible. 
    Let the instance be $(v_1(t),\dots,v_n(t))$ with $t\sim F_{\max}$.
    The expected payoff of an arbitrary stopping rule $\tau$ can be rewritten as follows: 
    \begin{align*}
        &\int_{\mathbb{R}}\sum_{i=1}^nv_i(t)\mathbb{P}(\tau=i | v_1(t),\dots,v_i(t))\mathrm{d}F_{\max}(t)&\text{(Definition expected payoff)}\\
        =&\sum_{j=1}^n\sum_{i=1}^n \int_{w_{j-1}}^{w_j}v_i(t)\mathbb{P}(\tau=i | v_1(t),\dots,v_i(t))\mathrm{d}F_{\max}(t)\\
        =&\sum_{j=1}^n\sum_{i=1}^j \int_{w_{j-1}}^{w_j}v_i(t)\mathbb{P}(\tau=i | v_1(t),\dots,v_i(t))\mathrm{d}F_{\max}(t)&(i>j\implies\text{ zero payoff})\\
        =&\sum_{i=1}^n \int_{w_{i-1}}^{w_i}v_i(t)\mathbb{P}(\tau=i | v_1(t),\dots,v_i(t))\mathrm{d}F_{\max}(t)\\
        +&\sum_{i=1}^{n-1}\sum_{j=i+1}^n \int_{w_{j-1}}^{w_j}v_i(t)\mathbb{P}(\tau=i | v_1(t),\dots,v_i(t))\mathrm{d}F_{\max}(t)&\text{(Sum decomposition)}\\
        =&\sum_{i=1}^n \int_{w_{i-1}}^{w_i}t\mathbb{P}(\tau=i | b_1,\dots,b_{i-1},t)\mathrm{d}F_{\max}(t)\\
        +&\sum_{i=1}^{n-1}\sum_{j=i+1}^n \int_{w_{j-1}}^{w_j}b_i\mathbb{P}(\tau=i | b_1,\dots,b_i)\mathrm{d}F_{\max}(t)&\text{(Definition }v_i(t))\\
        =&\sum_{i=1}^n \int_{w_{i-1}}^{w_i}t\mathbb{P}(\tau=i | b_1,\dots,b_{i-1},t)\mathrm{d}F_{\max}(t)\\
        +&\sum_{i=1}^{n-1}\mathbb{P}(\tau=i | b_1,\dots,b_i)\cdot b_i\mathbb{P}(X_{\max}\ge w_i).
    \end{align*}
    By construction, $\mathbb{P}(\tau=i|b_1,\dots,b_{i-1},t)\le \mathbb{P}(\tau=i|b_1,\dots,b_i)$ for $t\in[w_{i-1}, w_i]$, therefore,
    \begin{align*}
    &\sum_{i=1}^n \int_{w_{i-1}}^{w_i}t\mathbb{P}(\tau=i | b_1,\dots,b_{i-1},t)\mathrm{d}F_{\max}(t)+\sum_{i=1}^{n-1}\mathbb{P}(\tau=i | b_1,\dots,b_i)\cdot b_i\mathbb{P}(X_{\max}\ge w_i)\\
        \le&\sum_{i=1}^n \mathbb{P}(\tau=i | b_1,\dots,b_i)\int_{w_{i-1}}^{w_i}t\mathrm{d}F_{\max}(t)+\sum_{i=1}^{n-1}\mathbb{P}(\tau=i | b_1,\dots,b_i)\cdot b_i\mathbb{P}(X_{\max}\ge w_i)\\
        \le &\max_{1\le i \le n}\left\{\int_{w_{i-1}}^{w_i}t\mathrm{d}F_{\max}(t)\right\}+\max_{1\le i \le n-1}\left\{b_i\mathbb{P}(X_{\max}\ge w_i)\right\}\\
       \le &\max_{1\le i \le n}\left\{\int_{w_{i-1}}^{w_i}t\mathrm{d}F_{\max}(t)\right\}+\max_{1\le i \le n-1}\left\{b_i\mathbb{P}(X_{\max}\ge b_i)\right\} \ \ \ \ \ \  \ \ \ \ \ \ \  (\text{using } b_i \leq w_i)\\
        \le &\max_{1\le i \le n}\left\{\int_{w_{i-1}}^{w_i}t\mathrm{d}F_{\max}(t)\right\} + \Pi^*(F_{\max}).
    \end{align*}
    This bound on the expected payoff holds for an arbitrary stopping rule, so also an optimal stopping rule. Lastly, because $F_{\max}$ is continuous, there exists, by the intermediate value theorem, $w_i$'s such that $\int_{w_{i-1}}^{w_i}t\mathrm{d}F_{\max}(t)=\frac{1}{n}\mathbb{E}[X_{\max}]$, yielding the desired result.
\end{cleverproof}

\section{Stronger lower and upper bounds}
\label{sec:tight_bounds}
The results of Theorem \ref{thm:first_bounds} can be made tighter if $F_{\max}$ is sufficiently smooth. Specifically, we prove the following theorem.

\begin{theorem}
\label{thm:special-cases}
    Let $n \in \N$ and let $F_{\max}$ be a given probability distribution independent of $n$. Define
    $$
    \Pi^*(F_{\max}) = \max_{ p \geq 0} p \cdot\mathbb{P}(X_{\max}\ge p) = p^* \cdot\mathbb{P}(X_{\max}\ge p^*)
    $$
    as the worst-case expected value of threshold algorithm using monopoly price $p^*$, and 
    $$
    V^*(F_{\max}, n) = \sup_{\tau\in\mathcal{T}}\inf_{\boldsymbol{X}\in\mathcal{I}(F_{\max}, n)}\mathbb{E}[\boldsymbol{X}_{\tau}]
    $$
    the worst-case expected payoff of the optimal randomised stopping algorithm for sequences of $n$ (correlated) values whose maximum follows $F_{\max}$. 
    If $p^*$ is unique and $F_{\max}$ is three times continuously differentiable at $p^*$, then $$V^*(F_{\max}, n)=\Pi^*(F_{\max})+\Theta(n^{-2}).$$
\end{theorem}  
The above result generalises to settings where $X_{\max}$ depends on $n$ through scaling.

\begin{corollary}
    \label{corol:scaling}
    Let $F_{\mathrm{base}}$ be a distribution that satisfies the properties of Theorem \ref{thm:special-cases}.
    If
    $$
    \frac{X_{\max}}{\mathbb{E}[X_{\max}]}\sim F_{\mathrm{base}},
    $$
    that is, $X_{\max}$ depends on $n$ through scaling, but retains the same shape, then
    $$
    V^*(F_{\max},n)=\Pi^*(F_{\max})+\Theta\left(\frac{\mathbb{E}[X_{\max}]}{n^2}\right).
    $$
\end{corollary}
\begin{proof}
    As shorthand, let $s(n)=\mathbb{E}[X_{\max}]$.
    From the definition of $V^*$ it follows that it has homogeneity.
    \begin{align*}
        V^*(F_{\max}, n)&=\sup_{\tau\in\mathcal{T}}\inf_{\boldsymbol{X}\in\mathcal{I}(F_{\max}, n)}\mathbb{E}[\boldsymbol{X}_{\tau}]
        =\sup_{\tau\in\mathcal{T}}\inf_{\frac{1}{s(n)}\boldsymbol{X}\in\mathcal{I}(F_{\mathrm{base}}, n)}\mathbb{E}[\boldsymbol{X}_{\tau}]\\
        &=s(n)\sup_{\tau\in\mathcal{T}}\inf_{\frac{1}{s(n)}\boldsymbol{X}\in\mathcal{I}(F_{\mathrm{base}}, n)}\mathbb{E}\left[\frac{1}{s(n)}\boldsymbol{X}_{\tau}\right]
        =s(n)V^*(F_{\mathrm{base}},n).
    \end{align*}
    By Theorem \ref{thm:special-cases}, $V^*(F_{\mathrm{base}},n)=\Pi^*(F_{\mathrm{base}})+\Theta(n^{-2})$, so using the homogeneity we obtain that 
    $$V^*(F_{\max},n)=s(n)V^*(F_{\mathrm{base}},n)=s(n)\Pi^*(F_{\mathrm{base}})+\Theta\left(\frac{s(n)}{n^2}\right)=\Pi^*(F_{\max})+\Theta\left(\frac{\mathbb{E}[X_{\max}]}{n^2}\right).\qedhere$$
\end{proof}
In the above proof it is not important that $F_{\mathrm{base}}$ does not depend on $n$, only that the assumptions of Theorem \ref{thm:special-cases} are satisfied.
The result can therefore be generalised to $F_{\mathrm{base},n}$ that depend on $n$, by having the $F_{\mathrm{base},n}$ satisfy the conditions of the theorem in a uniform sense.\footnotemark
\footnotetext{
We require that the distributional constants in Lemmas \ref{lemma:conditional-lowerbound} and \ref{lemma:conditional-upperbound} are uniformly bounded, for example, the $C(F_{\mathrm{base},n})$ are uniformly bounded.}

Theorem \ref{thm:special-cases} follows from Lemmas \ref{lemma:conditional-lowerbound} and \ref{lemma:conditional-upperbound}.
Both use the following distributional constant
$$
    C(F_{\max})=\frac{(1-F_{\max}(p^*))^{2}}{2F_{\max}'(p^*)+p^*F_{\max}''(p^*)}.
$$
\noindent Because $p^*$ is the unique maximiser, the second order condition $\Pi''(p) < 0$ informs us that $2F'_{\max}(p^*)+p^*F''_{\max}(p^*)>0$, so $C(F_{\max})$ is positive and finite. We now state the two lemmas.

\begin{lemma}
    \label{lemma:conditional-lowerbound}
        If $p^*$ is unique, $\epsilon=\frac{3(1-F_{\max}(p^*))}{2F_{\max}'(p^*)+pF_{\max}''(p^*)}\frac{1}{n}<1$, and the third derivative of $F_{\max}$ is bounded on $(p^*-\epsilon, p^*+\epsilon)$, then
        $$
            V^*(F_{\max}, n)\ge \Pi^*(F_{\max})+\frac{3}{2}C(F_{\max})n^{-2}+O(n^{-3}).
        $$
\end{lemma}

\begin{lemma}
    \label{lemma:conditional-upperbound}
        If $p^*$ is unique and $F_{\max}$ is three times continuously differentiable at $p^*$, then there are constants $C_1$ and $C_2$ which depend on $F_{\max}$, such that for $n> \max\{C_1,C_2\}$, 
        $$
            V^*(F_{\max}, n)\le \Pi^*(F_{\max})+2\pi^2C(F_{\max})(n-C_2)^{-2}+O(n^{-3}).
        $$
\end{lemma}

Lemmas \ref{lemma:conditional-lowerbound} and \ref{lemma:conditional-upperbound} both rely on Taylor's theorem, so we state it here for completeness.
\begin{lemma}[Taylor's theorem; \cite{rudin1976principles}]
    Let $f$
    be a real-valued function that is $n$ times differentiable on the interval $[a,b]$ and suppose $\alpha,\beta\in[a,b]$. Then there exists a point $x$ between $\alpha$ and $\beta$ such that
    $$
    f(\beta)=\sum_{k=0}^{n-1}\frac{f^{(k)}(\alpha)}{k!}(\beta-\alpha)^k+\frac{f^{(n)}(x)}{n!}(\beta-\alpha)^n.
    $$
\end{lemma}
\subsection{Lower bound}
What is remarkable about Lemma \ref{lemma:conditional-lowerbound} is that it achieves the optimal asymptotic order using a simple uniform random threshold policy. In Lemma \ref{lemma:uniform-lowerbound-maximizers} we identify the optimal instance against this policy using a geometric argument, circumventing difficult first order conditions. 
The uniformity of the policy allows for this argument.
Lemma \ref{lemma:expected-payoff-uniform-point-policy} provides a lower bound on the expected payoff for this optimal instance. Lemma \ref{lemma:conditional-lowerbound} then follows by applying Lemma \ref{lemma:expected-payoff-uniform-point-policy} alongside Taylor's theorem.

Let us start with the crucial result, the optimal instance against a uniformly random policy.
\begin{lemma}
    \label{lemma:uniform-lowerbound-maximizers}
    For a random threshold policy $\tau$ with a threshold uniformly distributed on $(t_1, t_2)$, the tuple $(v_1,\dots,v_n)$ that minimises the payoff, given $v_{\max}$, is
    \begin{itemize}
        \item $v_1=v_{\max}$ and $v_i=0$ for $i>1$ if $v_{\max}<t_1$;
        \item $v_i=t_1+\frac{i}{n}(v_{\max}-t_1)$ if $t_1\le v_{\max}\le t_2+\frac{t_2-t_1}{n-1}$;
        \item $v_i=t_1+\frac{i}{n-1}(t_2-t_1)$ for $i<n$ and $v_n=v_{\max}$ if $v_{\max}> t_2+\frac{t_2-t_1}{n-1}$. 
    \end{itemize}
\end{lemma}
\begin{proof}
    Denote the CDF of the random threshold by $F_T$. The payoff given a tuple $\boldsymbol{v}$ is
    \begin{align*}
        \boldsymbol{v}_{\tau}&=v_1\cdot\mathbb{P}(T\le v_1)+\sum_{i=2}^nv_i\mathbb{P}(v_{i-1}<T\le v_i)=v_1F_T(v_1)+\sum_{i=2}^nv_i(F_T(v_i)-F_T(v_{i-1})).
    \end{align*}
    If $v_{\max}<t_1$ setting $v_1=v_{\max}$ and $v_i=0$ for $i>1$ results in an expected payoff of zero, so it is optimal. For the other two cases Lemmas \ref{lemma:increasing-threshold} and \ref{lemma:support-restriction} inform us that there is an optimal tuple $\boldsymbol{v}$ which satisfies $t_1\le v_1<v_2<\dots<v_{n-1}\le t_2$. The expected payoff can therefore be visualised as in Figure \ref{fig:uniform-threshold-payoff}, where it is the sum of the shaded and striped area.
    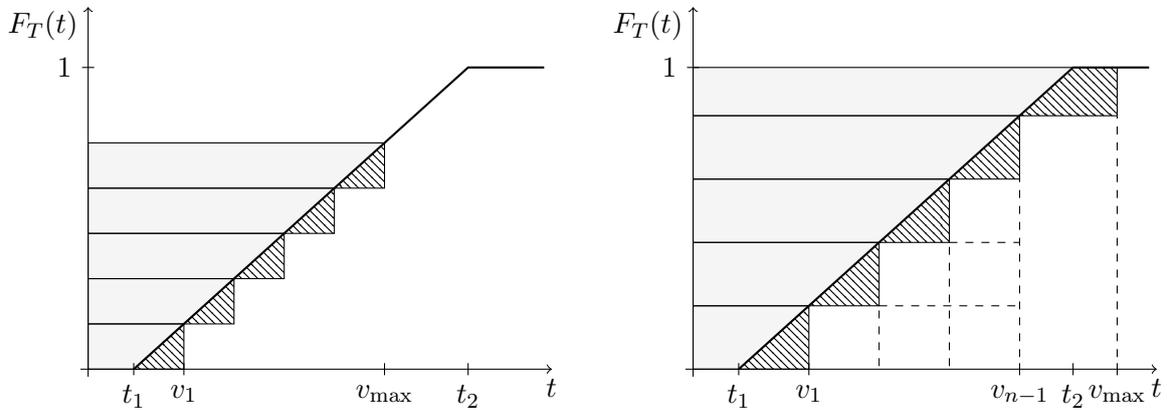
\begin{figure}[!ht]
        \begin{minipage}[b]{0.5\textwidth}
            \centering
            \begin{tikzpicture}[scale = 4]
                \draw[->] (0,-0.025) -- (0,1.2) node[pos=0.95, left] {$F_T(t)$};
                \draw[->] (-0.025,0) -- (1.525,0) node[below] {$t$};
                \draw (0.02,1) -- (-0.02,1) node[left] {$1$};
                \def\qWidth{0.15} 
                \foreach \i in {0, 1, 2, 3, 4} {
                    \fill[pattern=north west lines] (0.15 + 1.1 * \i * \qWidth + 1.1 * \qWidth, \i * \qWidth + \qWidth) -- (0.15 + 1.1 * \i * \qWidth, \i * \qWidth) -- (0.15 + 1.1 * \i * \qWidth + 1.1 * \qWidth, \i * \qWidth) --cycle;
            
                    \fill[gray!50, opacity=0.15] (0,\i * \qWidth) -- (0,\i * \qWidth + \qWidth) -- (0.15 + 1.1 * \i * \qWidth + 1.1 * \qWidth, \i * \qWidth + \qWidth) -- (0.15 + 1.1 * \i * \qWidth, \i * \qWidth) -- cycle;
                
                    \draw[black] (0,\i * \qWidth) -- (0,\i * \qWidth + \qWidth) -- (0.15 + 1.1 * \i * \qWidth + 1.1 * \qWidth, \i * \qWidth + \qWidth) -- (0.15 + 1.1 * \i * \qWidth + 1.1 * \qWidth, \i * \qWidth) -- cycle;
                }    
                \draw[thick] (0.15,0) -- (1.25,1) -- (1.5,1);
                \draw (0.15 + 1.1 * 1 * \qWidth,0.02) -- (0.15 + 1.1 * 1 * \qWidth,-0.02) node[below] {$v_1$};
                \draw (0.15 + 1.1 * 5 *  \qWidth,0.02) -- (0.15 + 1.1 * 5 * \qWidth,-0.02) node[below] {$v_{\max}$};
                \draw (0.15,0.02) -- (0.15,-0.02) node[below] {$t_1$};
                \draw (1.25,0.02) -- (1.25,-0.02) node[below] {$t_2$};
            \end{tikzpicture}
        \end{minipage}%
        \begin{minipage}[b]{0.5\textwidth}
            \centering
            \begin{tikzpicture}[scale = 4]
                \draw[->] (0,-0.025) -- (0,1.2) node[pos=0.95, left] {$F_T(t)$};
                \draw[->] (-0.025,0) -- (1.525,0) node[below] {$t$};
                \draw (0.02,1) -- (-0.02,1) node[left] {$1$};
            
                \def\qWidth{0.21} 
                \foreach \i in {0, 1, 2, 3} {
                    \fill[pattern=north west lines] (0.15 + 1.1 * \i * \qWidth + 1.1 * \qWidth, \i * \qWidth + \qWidth) -- (0.15 + 1.1 * \i * \qWidth, \i * \qWidth) -- (0.15 + 1.1 * \i * \qWidth + 1.1 * \qWidth, \i * \qWidth) --cycle;
            
                    \fill[gray!50, opacity=0.15] (0,\i * \qWidth) -- (0,\i * \qWidth + \qWidth) -- (0.15 + 1.1 * \i * \qWidth + 1.1 * \qWidth, \i * \qWidth + \qWidth) -- (0.15 + 1.1 * \i * \qWidth, \i * \qWidth) -- cycle;
                
                    \draw[black] (0,\i * \qWidth) -- (0,\i * \qWidth + \qWidth) -- (0.15 + 1.1 * \i * \qWidth + 1.1 * \qWidth, \i * \qWidth + \qWidth) -- (0.15 + 1.1 * \i * \qWidth + 1.1 * \qWidth, \i * \qWidth) -- cycle;
                }    
                \def\vMax{1.395} 
                \fill[pattern=north west lines] (1.25, 1) -- (0.15 + 4.4 * \qWidth, 4 * \qWidth) -- (\vMax, 4 * \qWidth) -- (\vMax, 1)-- cycle;
                
                \fill[gray!50, opacity=0.15] (0, 4 * \qWidth) -- (0, 1) -- (1.25, 1) -- (0.15 + 4.4 * \qWidth, 4 * \qWidth) -- cycle;
                
                \draw[black] (0, 4 * \qWidth) -- (0,1) -- (\vMax, 1) -- (\vMax, 4 * \qWidth) -- cycle;
                \draw[thick] (0.15,0) -- (1.25,1) -- (1.5,1);
                
                \draw[dashed] (0.15 + 1.1 * 2 * \qWidth, 0.01) -- (0.15 + 1.1 * 2 * \qWidth, 1 * \qWidth);
                \draw[dashed] (0.15 + 1.1 * 3 * \qWidth, 0.01) -- (0.15 + 1.1 * 3 * \qWidth, 2 * \qWidth);
                \draw[dashed] (0.15 + 1.1 * 4 * \qWidth, 0.04) -- (0.15 + 1.1 * 4 * \qWidth, 3 * \qWidth);
                \draw[dashed] (\vMax, 0.04) -- (\vMax, 4 * \qWidth);
                \draw[dashed] (0.15 + 1.1 * 2 * \qWidth, 1 * \qWidth) -- (0.15 + 1.1 * 4 * \qWidth, 1 * \qWidth);
                \draw[dashed] (0.15 + 1.1 * 3 * \qWidth, 2 * \qWidth) -- (0.15 + 1.1 * 4 * \qWidth, 2 * \qWidth);
                \draw (0.15 + 1.1 * 1 * \qWidth,0.02) -- (0.15 + 1.1 * 1 * \qWidth,-0.02) node[below] {$v_1$};
                \draw (0.15 + 1.1 * 4 * \qWidth,0.02) -- (0.15 + 1.1 * 4 * \qWidth,-0.02) node[below] {$v_{n-1}$};
                \draw (\vMax,0.02) -- (\vMax,-0.02) node[below] {$v_{\max}$};
                \draw (0.15,0.02) -- (0.15,-0.02) node[below] {$t_1$};
                \draw (1.25,0.02) -- (1.25,-0.02) node[below] {$t_2$};
            \end{tikzpicture}
        \end{minipage}
        \caption{Payoff visualisation for $v_{\max}\le t_2$ (left) and $v_{\max}>t_2$ (right) with $n=5$.}
        \label{fig:uniform-threshold-payoff}
    \end{figure}
    
    Observe that for a given $v_{\max}$ the $v_i$'s only influence the striped area. If $v_{\max}\le t_2$ the total area of these triangles is minimised if the $v_i$'s for $i<n$ are equally spaced between $t_1$ and $v_{\max}$, that is, $v_i=t_1+\frac{i}{n}(v_{\max}-t_1)$. 

    If $v_{\max}> t_2$ the striped area is one quadrilateral and $n-1$ triangles. Given $v_{n-1}$ the area of the triangles is minimised if the other $v_i$'s have an equal spacing, that is, $v_i=t_1+\frac{i}{n-1}(v_{n-1}-t_1)$ for $i<n$. The total area can therefore be expressed as
    $$
    \underbrace{1\cdot v_{\max} }_{\mathclap{\text{bounding box}}}-
    \underbrace{F_{T}\left(v_{n-1}\right)\left(v_{\max}-v_{n-1}\right)}_{\mathclap{\text{area under quadrilateral}}}-
    \underbrace{\binom{n-1}{2}\frac{F_{T}\left(v_{n-1}\right)}{n-1}\frac{v_{n-1}-t_{1}}{n-1}}_{\mathclap{\text{area of the }\binom{n-1}{2}\text{ rectangles under the triangles}}},
    $$
    where $\frac{F_{T}\left(v_{n-1}\right)}{n-1}$ and $\frac{v_{n-1}-t_{1}}{n-1}$ are the height and width of the rectangles under the triangles.
    Minimising the total area is therefore a univariate quadratic problem in $v_{n-1}$ with the constraint $v_{n-1}\le t_2$.
    
    It can be verified that the minimiser of this constrained quadratic is $v_{n-1}=\min\{t_2,t_1+\frac{n-1}{n}(v_{\max}-t_1)\}$. Plugging this in yields the desired result for the last two cases.
\end{proof}
We now compute and give a bound on the expected payoff.

\begin{lemma}
    \label{lemma:expected-payoff-uniform-point-policy}
    For a random threshold policy $\tau$ with a threshold uniformly distributed on $(t_1, t_2)$
    \begin{align*}
        \min_{\mathclap{\boldsymbol{X}\in\mathcal{I}(F_{\max},n)}}\mathbb{E}[\boldsymbol{X}_{\tau}]
        &=\mathbb{E}[f_1(X_{\max})\indicator{t_1 \le X_{\max}\le t_2}]+\mathbb{E}\left[f_2(X_{\max})\indicator{t_2 < X_{\max}\le t_2+\frac{t_2-t_1}{n-1}}\right]\\
        &+f_2\left(t_2+\frac{t_2-t_1}{n-1}\right)\mathbb{P}\left(X_{\max}> t_2+\frac{t_2-t_1}{n-1}\right)\\
        &\ge\mathbb{E}[f_1(X_{\max})\indicator{t_1 \le X_{\max}\le t_2}]+f_2(t_2)\mathbb{P}\left(X_{\max}>t_2\right),
    \end{align*}
    where
    \begin{align*}
        f_1(s)=\frac{s-t_1}{t_2-t_1}\left(t_1+\frac{n+1}{2n}(s-t_1)\right)\text{ and }
        f_2(s)=s-\frac{n-1}{2n}\frac{(s-t_1)^2}{t_2-t_1}.
    \end{align*}
\end{lemma}
\begin{proof}
    Denote the optimal $v_i$'s from Lemma \ref{lemma:uniform-lowerbound-maximizers}, given $v_{\max}$, by $\boldsymbol{v}^*(v_{\max})$. 
    If $v_{\max}<t_1$ then, as noted before, $\boldsymbol{v}_{\tau}^*(v_{\max})=0$. If $v_{\max}\in[t_1,t_2]$ then the payoff can be determined using the left visual of Figure \ref{fig:uniform-threshold-payoff}.
    $$
    \boldsymbol{v}_{\tau}^*(v_{\max})=\frac{v_{\max}-t_1}{t_2-t_1}\left(v_{\max}-\frac{n(n-1)}{2}\frac{1}{n}\frac{v_{\max}-t_1}{n}\right)=f_1(v_{\max}).
    $$
    If $v_{\max}\in(t_2,t_2+\frac{t_2-t_1}{n-1}]$ then the payoff can be determined using the right visual of Figure \ref{fig:uniform-threshold-payoff}.
    $$
    \boldsymbol{v}_{\tau}^*(v_{\max})=v_{\max}-\frac{v_{\max}-t_1}{t_2-t_1}\frac{n(n-1)}{2}\frac{1}{n}\frac{v_{\max}-t_1}{n}=f_2(v_{\max}).
    $$
    Lastly, if $v_{\max}>t_2+\frac{t_2-t_1}{n-1}$ 
    then $\boldsymbol{v}_{\tau}^*(v_{\max})=\boldsymbol{v}_{\tau}^*(t_2)$, since then $v_{n-1}=t_2$, so one of the first $n-1$ values is accepted.
    Taking the expectation yields the exact result. 
    
    The lower bound result follows from the fact that $f_2$ is an increasing function (higher $v_{\max}$ means a higher payoff). Therefore, 
    \begin{align*}
        &\mathbb{E}\left[f_2(X_{\max})\indicator{t_2 \le X_{\max}< t_2+\frac{t_2-t_1}{n-1}}\right]+\:f_2\left(t_2+\frac{t_2-t_1}{n-1}\right)\mathbb{P}\left(X_{\max}\ge t_2+\frac{t_2-t_1}{n-1}\right)\\
        \ge&\: \mathbb{E}\left[f_2(t_2)\indicator{t_2 \le X_{\max}< t_2+\frac{t_2-t_1}{n-1}}\right]
        +\:f_2\left(t_2\right)\mathbb{P}\left(X_{\max}\ge t_2+\frac{t_2-t_1}{n-1}\right)\\
        =\:&f_2(t_2)\mathbb{P}\left(X_{\max}\ge t_2\right). \qedhere
    \end{align*}
\end{proof}

We obtain an informative lower bound by using an appropriate Taylor series to the above result.
\begin{cleverproof}{lemma:conditional-lowerbound}
    Consider a random threshold that is uniform on $(p^*-\epsilon,p^*+\epsilon)$. Then the functions $f_i$ as defined in Lemma \ref{lemma:expected-payoff-uniform-point-policy} simplify to
    \begin{align*}
        f_1(p^*+s)&=\frac{s+\epsilon}{2\epsilon}\left(p^*-\epsilon+\frac{n+1}{2n}(s+\epsilon)\right),\\
        &=\underbrace{\frac{1}{2}p^*-\frac{1}{2}\epsilon+\frac{n+1}{4n}\epsilon}_{a(\epsilon)}+\underbrace{\left(\frac{p^*}{2\epsilon}-\frac{1}{2}+\frac{n+1}{2n}\right)}_{b(\epsilon)}s+\underbrace{\frac{n+1}{4n\epsilon}}_{c(\epsilon)}s^2\\
        f_2(p^*+\epsilon)&=p^*+\frac{\epsilon}{n}.
    \end{align*}
    Moreover, Lemma \ref{lemma:expected-payoff-uniform-point-policy} informs us that 
    \begin{align}
        V^*(F_{\max}, n)&\ge\mathbb{E}[f_1(X_{\max})\indicator{p^*-\epsilon\le X_{\max}\le p^*+\epsilon}]+f_2(p^*+\epsilon)\mathbb{P}\left(X_{\max}> p^*+\epsilon\right) \notag \\
        &=\int_{-\epsilon}^\epsilon (a(\epsilon)+b(\epsilon)s+c(\epsilon)s^2)F_{\max}'(p^*+s)\mathrm{d}s+\left(p^*+\frac{\epsilon}{n}\right)\left(1-F_{\max}(p^*+\epsilon)\right) \label{eq:lowerbound-terms}.
    \end{align}
    Now, $F'''_{\max}(p^*+s)$ exists for $|s|\le \epsilon$, because it is bounded, 
    so by Taylor's Theorem
    $$F'_{\max}(p^*+s)=F'_{\max}(p^*)+F''_{\max}(p^*)s+\frac{1}{2}F'''_{\max}(p^*+\xi(s))s^2,$$
    where $\xi(s)\in(0,s)$. Let $M=\sup\{|F'''_{\max}(p^*+s)|:|s|\le\epsilon\}$ and notice that $M\ge\sup\{|F'''_{\max}(p^*+\xi(s))|:|s|\le\epsilon\}$ and that $M$ is bounded.    
    Substituting the Taylor expansion into the first term of (\ref{eq:lowerbound-terms}) yields
    \begin{align*}
        &\int_{-\epsilon}^\epsilon (a(\epsilon)+b(\epsilon)s+c(\epsilon)s^2)F_{\max}'(p^*+s)\mathrm{d}s\\
        =&\int_{-\epsilon}^\epsilon (a(\epsilon)+b(\epsilon)s+c(\epsilon)s^2)(F'_{\max}(p^*)+F''_{\max}(p^*)s)\mathrm{d}s\\
        +&\int_{-\epsilon}^\epsilon (a(\epsilon)+b(\epsilon)s+c(\epsilon)s^2)\left(\frac{1}{2}F'''_{\max}(p^*+\xi(s))s^2\right)\mathrm{d}s.
    \end{align*}
    Integration reveals that the first term is equal to
    \begin{align}
        &2a(\epsilon)F'_{\max}(p^*)\epsilon+\frac{2}{3}\left(c(\epsilon)F'_{\max}(p^*)+b(\epsilon)F''_{\max}(p^*)\right)\epsilon^3 \notag \\
        =&p^*F'_{\max}(p^*)\epsilon+\left(\frac{2-n}{3n}F'_{\max}(p^*)+\frac{1}{3}p^*F''_{\max}(p^*)\right)\epsilon^2+O(\epsilon^3) \label{eq:lowerbound-terms-a}.
    \end{align}
    The absolute value of the second term can be bounded as follows
    \begin{align*}
        &\left|\int_{-\epsilon}^\epsilon (a(\epsilon)+b(\epsilon)s+c(\epsilon)s^2)\left(\frac{1}{2}F'''_{\max}(p^*+\xi(s))s^2\right)\mathrm{d}s\right|\\
        \le&\int_{-\epsilon}^\epsilon (|a(\epsilon)|+|b(\epsilon)|\cdot |s|+|c(\epsilon)|s^2)\left|\frac{1}{2}F'''_{\max}(p^*+\xi(s))\right|s^2\mathrm{d}s\\
        \le&\frac{M}{2}\int_{-\epsilon}^\epsilon (|a(\epsilon)|+|b(\epsilon)|\cdot |s|+|c(\epsilon)|s^2)s^2\mathrm{d}s\\
        =&M\left(\frac{1}{3}|a(\epsilon)|\epsilon^3+\frac{1}{4}|b(\epsilon)|\epsilon^4+\frac{1}{5}|c(\epsilon)|\epsilon^5\right).
    \end{align*}
    The above is $O(\epsilon^3)$, since $\epsilon<1$, so it gets absorbed in the previous $O(\epsilon^3)$ term. 
    Substituting the Taylor expansion into the second term of (\ref{eq:lowerbound-terms}) yields
    \begin{align}
        &\left(p^*+\frac{\epsilon}{n}\right)\left(1-F_{\max}(p^*+\epsilon)\right)\notag\\
        =&\left(p^*+\frac{\epsilon}{n}\right)\left(1-F_{\max}(p^*)- F_{\max}'(p^*)\epsilon-\frac{1}{2}F_{\max}''(p^*)\epsilon^2+O(\epsilon^3)\right)\notag\\
        =&\Pi^*(F_{\max})-p^*(F_{\max}'(p^*)\epsilon+\frac{1}{2}F''_{\max}(p^*)\epsilon^2)+\frac{\epsilon}{n}(1-F_{\max}(p^*)-F_{\max}'(p^*)\epsilon)+O(\epsilon^3) \label{eq:lowerbound-terms-b}.
    \end{align}
    The sum of (\ref{eq:lowerbound-terms-a})
and (\ref{eq:lowerbound-terms-b}) gives the asymptotics of (\ref{eq:lowerbound-terms}),
    $$
    \Pi^*(F_{\max})+\frac{1-F_{\max}(p^*)}{n}\epsilon-\frac{1}{6}\left(\frac{2n+2}{n}F_{\max}'(p^*)+p^*F_{\max}''(p^*)\right)\epsilon^2+O(\epsilon^3).
    $$
    The above expression has a maximum at $\epsilon^*=\frac{3(1-F_{\max}(p^*))}{(2F_{\max}'(p^*)+pF_{\max}''(p^*))n+2F_{\max}'(p^*)}$, but for simplicity we pick our $\epsilon=\frac{3(1-F_{\max}(p^*))}{2F_{\max}'(p^*)+pF_{\max}''(p^*)}\frac{1}{n}$. Substituting this $\epsilon$ into the above expression gives the result of the theorem. 
\end{cleverproof}
\subsection{Upper bound}
Lemma \ref{lemma:conditional-upperbound} relies on the same instance as Theorem \ref{thm:randomised-bound-continuous-xmax} and with it we show in Lemma \ref{lemma:w_(i+1)*P(X>=w_i)-bound} that for a smart choice of $w_{n-1}$ that $V^*(F_{\max},n)\le \max_{0\le i\le n-2}\{w_{i+1}\mathbb{P}(X_{\max}\ge w_i)\}$. Iteratively setting $w_{i+1}$ such that $w_{i+1}\mathbb{P}(X_{\max}\ge w_i)=\Pi^*(F_{\max})+\epsilon$ until $w_{i+1}$ surpasses $w_{n-1}$ ensures that $V^*(F_{\max},n)\le\Pi^*(F_{\max})+\epsilon$. Lemma \ref{lemma:conditional-upperbound} informs us which $\epsilon$ is possible for a given $n$. The technical arguments of Lemma \ref{lemma:conditional-upperbound} are in Appendix \ref{proofs:technical-lemmas}. 

Let us first derive a general upper bound from Theorem \ref{thm:randomised-bound-continuous-xmax}.
\begin{lemma}
\label{lemma:w_(i+1)*P(X>=w_i)-bound}
    Let $w_0$ and $w_n$ denote the lower and upper bound of the support of $X_{\max}$ then for any sequence $w_0<w_1<\dots<w_n$,
    $$V^*(F_{\max}, n)\le\max\left\{w_1\mathbb{P}(X_{\max}\ge w_0),\dots,w_{n-1}\mathbb{P}(X_{\max}\ge w_{n-2}),\int_{w_{n-1}}^\infty t\mathrm{d}F_{\max}(t)\right\}.$$
\end{lemma}
\begin{proof}
    Let $b_i$ be defined as in Theorem \ref{thm:randomised-bound-continuous-xmax}. An intermediary result of Theorem \ref{thm:randomised-bound-continuous-xmax} is that the expected payoff for a stopping rule $\tau$ has an upper bound of
    \begin{align*}
        \sum_{i=1}^n\mathbb{P}(\tau=i | b_1,\dots,b_i) \int_{w_{i-1}}^{w_i}t\mathrm{d}F_{\max}(t)+\sum_{i=1}^{n-1}\mathbb{P}(\tau=i | b_1,\dots,b_i)\cdot b_i\mathbb{P}(X_{\max}\ge w_i).
    \end{align*}
    Bounding the above in terms independent of $\tau$ yields an upper bound on $V^*(F_{\max}, n)$.
    As shorthand let $p_i=\mathbb{P}(\tau=i | b_1,\dots,b_i)$, then 
    \begin{align*}
        &\sum_{i=1}^np_i \int_{w_{i-1}}^{w_i}t\mathrm{d}F_{\max}(t)+\sum_{i=1}^{n-1}p_ib_i\mathbb{P}(X_{\max}\ge w_i)\\
        =&\,p_n\int_{w_{n-1}}^\infty t\mathrm{d}F_{\max}(t)+\sum_{i=1}^{n-1}p_i\left( \int_{w_{i-1}}^{w_i}t\mathrm{d}F_{\max}(t)+b_i\mathbb{P}(X_{\max}\ge w_i)\right)\\
        \le &\,p_n\int_{w_{n-1}}^\infty t\mathrm{d}F_{\max}(t)+\sum_{i=1}^{n-1}p_i\left( \int_{w_{i-1}}^{w_i}w_i\mathrm{d}F_{\max}(t)+w_i\mathbb{P}(X_{\max}\ge w_i)\right)\\
        =&\,p_n\int_{w_{n-1}}^\infty t\mathrm{d}F_{\max}(t)+\sum_{i=1}^{n-1}p_iw_i\mathbb{P}(X_{\max}\ge w_{i-1})\\
        \le& \,\max\left\{w_1\mathbb{P}(X_{\max}\ge w_0),\dots,w_{n-1}\mathbb{P}(X_{\max}\ge w_{n-2}),\int_{w_{n-1}}^\infty t\mathrm{d}F_{\max}(t)\right\}. \qedhere
    \end{align*}
\end{proof}

To prove the theorem we use the procedure laid out. It is a simple concept, but to turn it into an upper bound requires non-standard calculus arguments. 
\begin{cleverproof}{lemma:conditional-upperbound}
    Let $\Pi^*=\Pi^*(F_{\max})$.
    Define $\beta=\inf\{s:\int_s^\infty t\mathrm{d}F_{\max}(t)\le \Pi^*\}$. $\beta$ is finite and by choosing $w_{n-1}$ equal to $\beta$, Lemma \ref{lemma:w_(i+1)*P(X>=w_i)-bound} reduces to $V^*(F_{\max}, n)\le U(\boldsymbol{w}):=\max_{0\le i\le n-2}\{w_{i+1}\mathbb{P}(X_{\max}\ge w_i)\}$, since the upper bound is at least $\Pi^*$.

    We are interested in determining the smallest $\epsilon$ for which $U(\boldsymbol{w})\le \Pi^*+\epsilon$ for a fixed $n$. To do this we consider the inverse problem where $\epsilon$ is given and we want to determine the smallest $n$ such that $U(\boldsymbol{w})\le \Pi^*+\epsilon$.
    Consider iteratively setting 
    $$w_{i+1}=M(w_i;\epsilon):=\frac{\Pi^*+\epsilon}{\mathbb{P}(X_{\max}\ge w_i)}$$ 
    until a $w_{i+1}$ surpasses $\beta$, then by construction $U(\boldsymbol{w})= \Pi^*+\epsilon$.
    Define $f(t;\epsilon)=M(t;\epsilon)-t$, $M^k=M\circ M^{k-1}$ and $k(a,b;\epsilon)=\inf\{k:M^k(a;\epsilon)\ge b\}$.
    The quantity of interest is $k(0,\beta;\epsilon)$ by assuming without loss of generality that $w_0 = 0$. 

    Firstly, $k(0,\beta;\epsilon)$ is bounded and $O(\epsilon^{-1})$. 
    To see this, define $m(\epsilon)=\inf_{t\in[0,\beta]}\{f(t;\epsilon)\}=f(t^*(\epsilon);\epsilon)$, by the envelope theorem 
    $m'(\epsilon)=\frac{\partial}{\partial\epsilon}f(t;\epsilon)\big|_{t=t^*(\epsilon)}$, so $m'(0)=1/\mathbb{P}(X_{\max}\ge p^*)$, therefore by Taylor's theorem $m(\epsilon)=m(0)+m'(0)\epsilon+O(\epsilon^2)=\epsilon / \mathbb{P}(X_{\max}\ge p^*)+O(\epsilon^2)=\Theta(\epsilon)$. By construction, $\min_{0\le i\le n-2}\{w_{i+1}-w_i\}\ge m(\epsilon)$, so $k(0,\beta;\epsilon)\le\lceil\frac{\beta}{m(\epsilon)}\rceil=O(\epsilon^{-1})$.
    
    Secondly, observe that for $\epsilon'<\epsilon$ that $k(a,b;\epsilon') \ge k(a,b;\epsilon)$, lastly we have a triangle inequality $k(a,c;\epsilon)\le k(a,b;\epsilon) + k(b,c;\epsilon)$. We can therefore upper bound $k(0,\beta;\epsilon)$ by
    \begin{align}
        k(0,p^*-\delta;0)+k(p^*-\delta,p^*+\delta;\epsilon)+k(p^*+\delta,\beta;0),
        \label{eq:three-terms}
    \end{align}
    where $\delta>0$ will be chosen to have several nice properties. This is encapsulated in Lemma \ref{lemma:delta-choice} below.
    \begin{lemma}
        \label{lemma:delta-choice}
        Let $M(t;\epsilon):=\frac{\Pi^*+\epsilon}{\mathbb{P}(X_{\max}\ge t)}$, $f(t;\epsilon)=M(t;\epsilon)-t$ and $\bar{p}(\epsilon)$ the minimiser of $f$. 
        If $F_{\max}$ is continuously differentiable at $p^*$, then there exist $\epsilon_0'>0$ and $\delta>0$ (independent of $\epsilon$) such that for $\epsilon<\epsilon_0'$ the following holds: $M$ is increasing and continuous on $(p^*-\delta,p^*+\delta)$, and $f$ is non-negative, decreasing on $(p^*-\delta,\bar{p}(\epsilon)]$ and increasing on $[\bar{p}(\epsilon),p^*+\delta)$
    \end{lemma}
    \begin{proof}
        See Appendix \ref{proofs:technical-lemmas}.
    \end{proof}
    We have to choose $\delta$ in this way to later derive (\ref{eq:limit-result}), moreover $\delta$ can only be chosen this way if $\epsilon$ is small enough, therefore we require $n$ to be larger than some constant $C_1$. Now, the first and third terms of (\ref{eq:three-terms}) can be bounded by a simple argument. Observe that $f(t;0)>0$ if $|t-p^*|>\delta$ for $\delta>0$, so
    $$
    k(0,p^*-\delta;0)\le 1 + \frac{p^*-\delta}{\inf_{t\in[0,p^*-\delta]}\{f(t;\epsilon)\}}<\infty.
    $$
    Analogously, $k(p^*+\delta,\beta;0)$ can be bounded. 
    The second term of (\ref{eq:three-terms}) depends on $\epsilon$ and diverges as $\epsilon\to 0$, but we have the following result
    \begin{equation}
        \lim_{\epsilon\to 0}\epsilon^{\frac{1}{2}}k(p^*-\delta,p^*+\delta;\epsilon)\le\sqrt{2\pi^2\frac{(1-F_{\max}(p^*))^2}{p^*F''_{\max}(p^*)+2F'_{\max}(p^*)}}. \label{eq:limit-result}
    \end{equation}
    The proof of this statement is rather technical and therefore relegated to Lemma \ref{lemma:technical-upperbound-argument} in the appendix. As a consequence of the above result,
    $$
    k(0,\beta;\epsilon)\le C_2 + \sqrt{2}\pi\left(\frac{(1-F_{\max}(p^*))^2}{p^*F''_{\max}(p^*)+2F'_{\max}(p^*)}\right)^{\frac{1}{2}}\epsilon^{-\frac{1}{2}},
    $$
    where $C_2=k(0,p^*-\delta;0)+k(p^*+\delta,\beta;0)$ is a constant that does not depend on $n$.
    The above expression relates $n$, the left hand side, to the possible error. Rewriting yields that $n$ points can attain an error of
    $$
    2\pi^2\frac{(1-F_{\max}(p^*))^2}{2F'_{\max}(p^*)+p^*F''_{\max}(p^*)}(n-C_2)^{-2},
    $$ 
    which proves the theorem.
\end{cleverproof}

\section{Convergence and prophet inequalities}
\label{sec:prophets}

In this section we will do two things. First, we will walk through a numerical example to highlight our results in Section \ref{sec:first_bounds} and \ref{sec:tight_bounds}.

Secondly, we compare the payoff of the online policy to the payoff of the offline optimum, namely, the expected value of the maximum value. 
This comparison is captured by the ratio
$$
r(F_{\max},n)=\frac{V^*(F_{\max},n)}{\mathbb{E}[X_{\max}]}.
$$
Without any assumptions regarding $X_{\max}$ we cannot obtain meaningful results regarding the performance, because in \cite{hill1983stop} it is shown that for any $n$ there is a sequence of distributions $F_{k,n}$ such that $\lim_{k\to\infty}r(F_{k,n},n)= 1/n$. This demonstrates that the online policy can perform arbitrarily worse than the offline optimum.
We will however show that $r(F_{\max},n)$ has a lower bound independent of $n$ under a natural assumption on $X_{\max}$.

\subsection{Numerical example}
We will consider our results through the lens of an example. 
Let $X_{\max}=\sqrt{n}X_{\mathrm{base}}$, where $X_{\mathrm{base}}$ has a Fréchet distribution with a mean of $1$ and variance of $3$. The distribution function of $X_{\mathrm{base}}$ is therefore given by 
$
F_{\text{base}}(x)=\exp\left(-\left(x/s\right)^{-\gamma}\right),
$
with $s=0.613$ and $\gamma=2.197$ to have the prescribed mean and variance (in this section decimals are truncated to three places).

The choice of a sublinear mean of $X_{\max}$
is motivated by the idea that,
while each additional buyer added to the market, that is, the group of all potential buyers, may increase the maximum valuation, the larger the market becomes the less likely it is that a new buyer will significantly raise the current maximum. The Fréchet distribution was chosen, because it belongs to the family of generalised extreme value (GEV) distributions.
In extreme value theory, which concerns the study of the maximum of random variables, the GEV distribution is a key concept and is often used to model the distribution of maxima.
This intersection between extreme value theory and prophet inequality type problems has already been observed before by \cite{kennedy1991asymptotic,correa2021optimal,livanos2024minimization}.\medskip

\noindent The distribution $F_{\text{base}}$ has a unique monopoly price $p^*=0.524$ 
and is three times continuously differentiable at $p^*$.
$\Pi^*(F_{\max})=0.396\sqrt{n}$,
so by Corollary \ref{corol:scaling} $V^*(F_{\max},n)=0.396\sqrt{n}+\Theta(n^{-3/2})$. 
In Figure \ref{fig:bounds-computed} we have computed the precise bounds on $V^*(F_{\max},n)$ by computing the constants of Lemmas \ref{lemma:conditional-lowerbound} and \ref{lemma:conditional-upperbound}.\footnotemark 
\footnotetext{For Lemma \ref{lemma:conditional-lowerbound}, $\epsilon=0.826<1$. For Lemma \ref{lemma:conditional-upperbound}, $\beta=1.398$, there is no restriction on $\delta$, so $C_1=1$, for $\delta=0.1$ we find $C_2=10$. Lastly,
$C(F_{\mathrm{base}})=0.208$ for both lemmas.
}
Moreover, we have also computed the bound that results from Lemma \ref{lemma:w_(i+1)*P(X>=w_i)-bound} when the $w_i$'s are chosen iteratively.
\begin{figure}[!h]
    \centering
    \hspace{-0.7cm}
    \begin{tikzpicture}
    \definecolor{cbblue}{rgb}{0.00,0.75,0.75}
    \definecolor{cbred}{rgb}{0.84,0.37,0.00}
    \definecolor{cborange}{rgb}{0.90,0.60,0.00}
    \definecolor{cbgreen}{rgb}{0.00,0.62,0.45}
    \definecolor{cbpurple}{rgb}{0.80,0.47,0.65}
    \begin{groupplot}[
        group style={
            group size=2 by 1,
            horizontal sep=1cm,
        },
        axis x line=bottom,
        axis y line=left,
        xlabel={$k$},
        xlabel style={at={(axis description cs:0.5, -0.1)}, anchor=north},
        grid=none,
        width=8.5cm, height=7cm,
        legend style={
            at={(0.5,0.3)},
            anchor=west,
            cells={anchor=west}
        },
    ]
    
    \nextgroupplot[
        ylabel={$V^*(F_{\max},n)$},
        ylabel style={at={(axis description cs:-0.05,0.5)}},
        xlabel={$n$},
        xlabel style={
            at={(current axis.right of origin)},
            anchor=west
        },
        ymin=0, ymax=3,
        ytick={0.005, 1, 2, 2.995},
        yticklabels={0, 1, 2, 3},
        xtick={2.005,15, 30, 44.995},
        xticklabels={2, 15, 30, 45}
    ]
    \addplot[
        color=black, solid,
        ultra thick
    ] coordinates {
        (11, 14.949) (12, 4.933) (13, 3.076) (14, 2.445) (15, 2.172) (16, 2.043) (17, 1.981) (18, 1.955) (19, 1.949) (20, 1.957) (21, 1.973) (22, 1.994) (23, 2.018) (24, 2.045) (25, 2.074) (26, 2.103) (27, 2.134) (28, 2.165) (29, 2.196) (30, 2.228) (31, 2.259) (32, 2.291) (33, 2.322) (34, 2.353) (35, 2.384) (36, 2.415) (37, 2.446) (38, 2.476) (39, 2.506) (40, 2.536) (41, 2.566) (42, 2.595) (43, 2.625) (44, 2.653) (45, 2.682)
    };
    \addlegendentry{Lemma \ref{lemma:conditional-upperbound}}
    
    \addplot[
        color=black, dashed,
        thick
    ] coordinates {
        (2, 1.978) (3, 1.237) (4, 1.139) (5, 1.139) (6, 1.169) (7, 1.21) (8, 1.257) (9, 1.306) (10, 1.355) (11, 1.404) (12, 1.453) (13, 1.501) (14, 1.548) (15, 1.595) (16, 1.64) (17, 1.685) (18, 1.728) (19, 1.771) (20, 1.813) (21, 1.854) (22, 1.895) (23, 1.934) (24, 1.973) (25, 2.012) (26, 2.049) (27, 2.087) (28, 2.123) (29, 2.159) (30, 2.194) (31, 2.229) (32, 2.264) (33, 2.297) (34, 2.331) (35, 2.364) (36, 2.396) (37, 2.429) (38, 2.46) (39, 2.492) (40, 2.523) (41, 2.553) (42, 2.584) (43, 2.613) (44, 2.643) (45, 2.672) 
    };
    \addlegendentry{Lemma \ref{lemma:w_(i+1)*P(X>=w_i)-bound}}

    \addplot[
        color=black, dotted,
        thick
    ] coordinates {
        (2, 1.268) (3, 1.264) (4, 1.293) (5, 1.334) (6, 1.379) (7, 1.427) (8, 1.475) (9, 1.523) (10, 1.57) (11, 1.616) (12, 1.662) (13, 1.707) (14, 1.751) (15, 1.794) (16, 1.836) (17, 1.877) (18, 1.918) (19, 1.958) (20, 1.997) (21, 2.035) (22, 2.073) (23, 2.11) (24, 2.146) (25, 2.182) (26, 2.218) (27, 2.253) (28, 2.287) (29, 2.321) (30, 2.354) (31, 2.387) (32, 2.42) (33, 2.452) (34, 2.483) (35, 2.515) (36, 2.545) (37, 2.576) (38, 2.606) (39, 2.636) (40, 2.666) (41, 2.695) (42, 2.724) (43, 2.752) (44, 2.781) (45, 2.809) 
    };
    \addlegendentry{Theorem \ref{thm:randomised-bound-continuous-xmax}}

    \addplot[
        color=cbblue, solid,
        thick
    ] coordinates {
        (2, 0.695) (3, 0.756) (4, 0.837) (5, 0.918) (6, 0.995) (7, 1.067) (8, 1.136) (9, 1.202) (10, 1.264) (11, 1.324) (12, 1.381) (13, 1.437) (14, 1.49) (15, 1.541) (16, 1.591) (17, 1.639) (18, 1.686) (19, 1.732) (20, 1.777) (21, 1.82) (22, 1.863) (23, 1.904) (24, 1.945) (25, 1.985) (26, 2.024) (27, 2.062) (28, 2.1) (29, 2.137) (30, 2.174) (31, 2.209) (32, 2.245) (33, 2.279) (34, 2.313) (35, 2.347) (36, 2.38) (37, 2.413) (38, 2.445) (39, 2.477) (40, 2.509) (41, 2.54) (42, 2.571) (43, 2.601) (44, 2.631) (45, 2.661) 
    };
    \addlegendentry{Lemma \ref{lemma:conditional-lowerbound}}

    \nextgroupplot[
        ylabel={$\Theta(n^{-3/2})$ constant},
        ylabel style={at={(axis description cs:-0.05,0.5)}},
        xlabel={$n$},
        xlabel style={
            at={(current axis.right of origin)},
            anchor=west
        },
        ymin=0, ymax=5,
        ytick={0.005, 1, 2, 3, 4, 4.995},
        yticklabels={0, 1, 2, 3, 4, 5},
        xtick={300, 600, 899.995},
        xticklabels={300, 600, 900},
    ]
    \addplot[
        color=black, solid,
        ultra thick
    ] coordinates {
        (11, 497.402) (12, 147.987) (13, 77.191) (14, 50.357) (15, 36.997) (16, 29.232) (17, 24.245) (18, 20.811) (19, 18.321) (20, 16.443) (21, 14.982) (22, 13.817) (23, 12.867) (24, 12.081) (25, 11.419) (26, 10.855) (27, 10.369) (28, 9.947) (29, 9.577) (30, 9.249) (31, 8.958) (32, 8.697) (33, 8.462) (34, 8.25) (35, 8.057) (36, 7.881) (37, 7.72) (38, 7.571) (39, 7.435) (40, 7.308) (41, 7.191) (42, 7.081) (43, 6.98) (44, 6.884) (45, 6.795) (46, 6.712) (47, 6.633) (48, 6.559) (49, 6.489) (50, 6.423) (51, 6.361) (52, 6.301) (53, 6.245) (54, 6.192) (55, 6.141) (60, 5.919) (75, 5.473) (90, 5.203) (105, 5.022) (120, 4.892) (135, 4.795) (150, 4.719) (165, 4.658) (180, 4.609) (195, 4.567) (210, 4.532) (225, 4.502) (240, 4.476) (255, 4.453) (270, 4.433) (285, 4.415) (300, 4.399) (315, 4.385) (330, 4.372) (345, 4.36) (360, 4.349) (375, 4.339) (390, 4.33) (405, 4.322) (420, 4.314) (435, 4.306) (450, 4.3) (465, 4.293) (480, 4.288) (495, 4.282) (510, 4.277) (525, 4.272) (540, 4.267) (555, 4.263) (570, 4.259) (585, 4.255) (600, 4.251) (615, 4.248) (630, 4.244) (645, 4.241) (660, 4.238) (675, 4.235) (690, 4.233) (705, 4.23) (720, 4.227) (735, 4.225) (750, 4.223) (765, 4.22) (780, 4.218) (795, 4.216) (810, 4.214) (825, 4.212) (840, 4.21) (855, 4.209) (870, 4.207) (885, 4.205) (900, 4.204) 
    };
    \addlegendentry{Lemma \ref{lemma:conditional-upperbound}}
    
    \addplot[
        color=black, dashed,
        thick
    ] coordinates {
        (2, 4.008) (3, 2.861) (4, 2.771) (5, 2.823) (6, 2.905) (7, 2.989) (8, 3.068) (9, 3.141) (10, 3.205) (11, 3.263) (12, 3.314) (13, 3.36) (14, 3.403) (15, 3.44) (16, 3.474) (17, 3.505) (18, 3.532) (19, 3.558) (20, 3.582) (21, 3.604) (22, 3.624) (23, 3.643) (24, 3.659) (25, 3.676) (26, 3.691) (27, 3.704) (28, 3.717) (29, 3.73) (30, 3.741) (31, 3.752) (32, 3.763) (33, 3.772) (34, 3.782) (35, 3.791) (36, 3.799) (37, 3.807) (38, 3.814) (39, 3.821) (40, 3.829) (41, 3.835) (42, 3.841) (43, 3.847) (44, 3.852) (45, 3.858) (46, 3.864) (47, 3.868) (48, 3.872) (49, 3.878) (50, 3.882) (51, 3.887) (52, 3.89) (53, 3.894) (54, 3.899) (55, 3.903) (60, 3.919) (75, 3.956) (90, 3.982) (105, 3.999) (120, 4.014) (135, 4.025) (150, 4.033) (165, 4.04) (180, 4.046) (195, 4.05) (210, 4.055) (225, 4.059) (240, 4.062) (255, 4.066) (270, 4.067) (285, 4.07) (300, 4.072) (315, 4.073) (330, 4.075) (345, 4.077) (360, 4.079) (375, 4.08) (390, 4.081) (405, 4.083) (420, 4.083) (435, 4.084) (450, 4.086) (465, 4.086) (480, 4.086) (495, 4.088) (510, 4.088) (525, 4.088) (540, 4.09) (555, 4.089) (570, 4.09) (585, 4.09) (600, 4.091) (615, 4.093) (630, 4.092) (645, 4.094) (660, 4.093) (675, 4.093) (690, 4.094) (705, 4.095) (720, 4.095) (735, 4.095) (750, 4.096) (765, 4.095) (780, 4.096) (795, 4.097) (810, 4.097) (825, 4.097) (840, 4.098) (855, 4.098) (870, 4.097) (885, 4.098) (900, 4.099) 
    };
    \addlegendentry{Lemma \ref{lemma:w_(i+1)*P(X>=w_i)-bound}}    

    \addplot[
        color=cbblue, solid,
        thick
    ] coordinates {
        (2, 0.378) (3, 0.362) (4, 0.354) (5, 0.35) (6, 0.346) (7, 0.343) (8, 0.341) (9, 0.338) (10, 0.336) (11, 0.335) (12, 0.333) (13, 0.332) (14, 0.331) (15, 0.33) (16, 0.329) (17, 0.328) (18, 0.327) (19, 0.326) (20, 0.326) (21, 0.325) (22, 0.325) (23, 0.324) (24, 0.324) (25, 0.323) (26, 0.323) (27, 0.323) (28, 0.322) (29, 0.322) (30, 0.322) (31, 0.321) (32, 0.321) (33, 0.321) (34, 0.321) (35, 0.32) (36, 0.32) (37, 0.32) (38, 0.32) (39, 0.32) (40, 0.319) (41, 0.319) (42, 0.319) (43, 0.319) (44, 0.319) (45, 0.319) (46, 0.319) (47, 0.318) (48, 0.318) (49, 0.318) (50, 0.318) (51, 0.318) (52, 0.318) (53, 0.318) (54, 0.318) (55, 0.318) (60, 0.317) (75, 0.316) (90, 0.316) (105, 0.315) (120, 0.315) (135, 0.315) (150, 0.314) (165, 0.314) (180, 0.314) (195, 0.314) (210, 0.314) (225, 0.314) (240, 0.314) (255, 0.314) (270, 0.313) (285, 0.313) (300, 0.313) (315, 0.313) (330, 0.313) (345, 0.313) (360, 0.313) (375, 0.313) (390, 0.313) (405, 0.313) (420, 0.313) (435, 0.313) (450, 0.313) (465, 0.313) (480, 0.313) (495, 0.313) (510, 0.313) (525, 0.313) (540, 0.313) (555, 0.313) (570, 0.313) (585, 0.313) (600, 0.313) (615, 0.313) (630, 0.313) (645, 0.313) (660, 0.313) (675, 0.313) (690, 0.313) (705, 0.313) (720, 0.313) (735, 0.313) (750, 0.313) (765, 0.313) (780, 0.313) (795, 0.313) (810, 0.313) (825, 0.313) (840, 0.313) (855, 0.313) (870, 0.313) (885, 0.313) (900, 0.313) 
    };
    \addlegendentry{Lemma \ref{lemma:conditional-lowerbound}}

    \end{groupplot}
    \end{tikzpicture}
    \caption{Explicit lower (blue) and upper (black) bounds on $V^*(F_{\max},n)$ (left) and the constant associated with $\Theta(n^{-3/2})$ (right).}
    \label{fig:bounds-computed}
\end{figure}
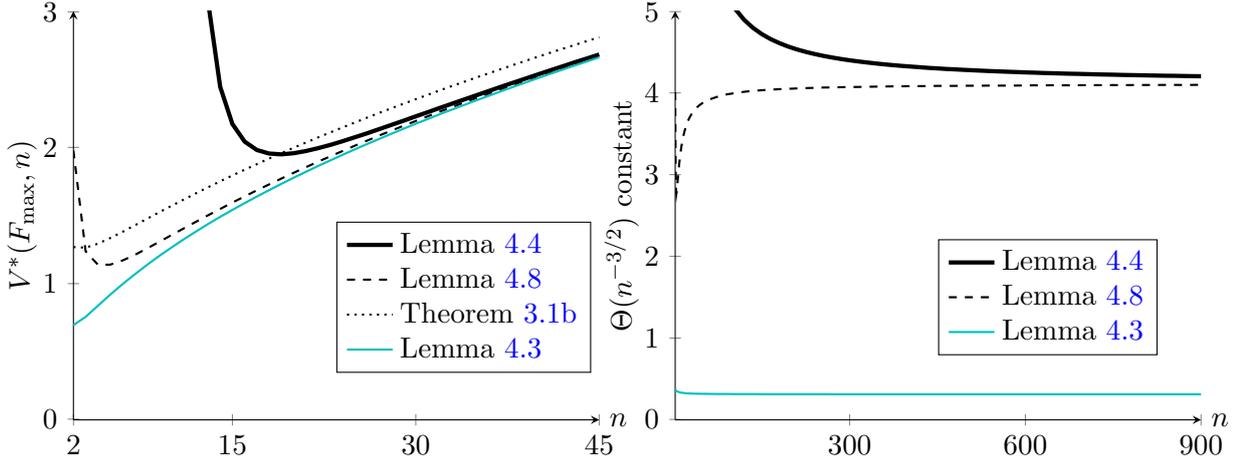

We see that the gap between the lower and upper bounds of $V^*(F_{\max},n)$ decreases very quickly, like we would expect from our results. We also observe that the constants associated with our lower and upper bound converge quickly, indicating that the asymptotic result kicks in early.

\subsection{Prophet inequality}
Because $V^*(F_{\max},n)$ is lower bounded by and converges to $\Pi^*(F_{\max})$, it suffices to consider the ratio $\Pi^*(F_{\max})/\mathbb{E}[X_{\max}]$ in order to establish a performance guarantee.

As noted before, no non-trivial performance guarantees are possible without additional assumptions due to \cite{hill1983stop} their result. However, the construction used to demonstrate this impossibility result is highly pathological.
This is because the coefficient of variation of the maximum valuation, which is given by
$
\mathrm{CV}(X_{\max})=\sqrt{\text{var}(X_{\max})}/\mathbb{E}[X_{\max}],
$ tends to infinity in the construction. The construction would therefore imply 
that the market becomes increasingly volatile as it grows, which contradicts the standard economic intuition that larger markets tend to be more stable. 
This therefore motivates the assumption that 
$\mathrm{CV}(X_{\max})$ does not grow too strongly.\medskip

\noindent The ratio between $\Pi^*(F_{\max})$ and $\mathbb{E}[X_{\max}]$ has already been studied under limited distributional information, such as a known mean and maximal dispersion; see also, e.g., \cite{van2024robust} for other choices of limited information. Tight (implicit) lower bounds for this ratio have been established in terms of the mean and variance. 
In particular, Lemma 2.3 in \cite{giannakopoulos2023robust} shows that this ratio is asymptotically bounded from below by $(1+\log(1+\mathrm{CV}(X_{\max})^2))^{-1}$. 
We extend this result by providing a non-asymptotic lower bound of the same order, leading to the following result on the performance of the monopoly price threshold policy.

\begin{theorem}
Let $n \in \N$ and $\mu,\sigma > 0$ (possibly depending on $n$). Suppose $F_{\max}$ has mean $\mu > 0$ and variance $\sigma^2$. Then for any $\beta > 1$, there exists a constant $c(\beta)$ such that 

$$
\inf_{\boldsymbol{X}\in\mathcal{I}(F_{\max}, n)}\sup_{\tau\in\mathcal{T}} \; \frac{\mathbb{E}[\boldsymbol{X}_{\tau}]}{\mathbb{E}[X_{\max}]} \geq  \left(c(\beta)+\beta\log \left(1+\frac{\sigma^2}{\mu^2}\right)\right)^{-1}.
$$
\label{thm:asymptotics-minimax-pricing}
\end{theorem}
\begin{proof}
    Let $\mathcal{P}(\mu,\sigma)$ denote the set of distributions with mean $\mu$ and variance $\sigma^2$. We have that
    \begin{align*}
        \inf_{\boldsymbol{X}\in\mathcal{I}(F_{\max}, n)}\sup_{\tau\in\mathcal{T}} \; \frac{\mathbb{E}[\boldsymbol{X}_{\tau}]}{\mathbb{E}[X_{\max}]}&\ge \inf_{\boldsymbol{X}\in\mathcal{I}(F_{\max}, n)} \frac{\mathbb{E}[\Pi^*(F_{\max})]}{\mathbb{E}[X_{\max}]}\\
        &= \inf_{\boldsymbol{X}\in\mathcal{I}(F_{\max}, n)}\sup_{p} \frac{p\cdot\mathbb{P}(X_{\max}\ge p)}{\mathbb{E}[X_{\max}]}\\
        &\ge \inf_{X_{\max}\in \mathcal{P}(\mu,\sigma)}\sup_{p}\frac{p\cdot\mathbb{P}(X_{\max}\ge p)}{\mathbb{E}[X_{\max}]}.
    \end{align*}
    In \cite{chen2024screening} an implicit expression for the final quantity is given under an additional support upper bound constraint. Letting this upper bound tend to infinity results in the following worst-case distribution for our problem
$$
F(x)=\begin{cases}
    0,&\text{if } x\le \pi,\\
    1-\frac{\pi}{x},&\text{if } \pi < x\le k,\\
    1,&\text{if } x>k,
\end{cases}
$$
where $\pi$ is the optimal value of $p\cdot\mathbb{P}(X_{\max}\ge p)$ and solves $\pi+\pi\log\left(\frac{k}{\pi}\right)=\mu$ and $k$ is such that the variance is $\sigma^2$. The former condition yields $k=\pi\exp\left(\frac{\mu}{\pi}-1\right)$
and the latter $2\pi k - \pi^2=\sigma^2+\mu^2$. 
Taken together we obtain the following implicit expression,
$$
2z^2\exp\left(z^{-1}-1\right)-z^2=1+\frac{\sigma^2}{\mu^2},
$$
where $z=\frac{\pi}{\mu}$ is the quantity of interest.
Let $g(t)= t^{-2}\left(2\exp\left(t-1\right)-1\right)$. Observe that our implicit expression becomes $g(1/z)=1+\sigma^2/\mu^2$.

Now, for $x\ge 0$, $g^{-1}(e^x)$ is concave. This can be shown as follows: $g^{-1}(e^x)$ is concave for $x\ge 0$ if $\{(x,y)\in\mathbb{R}_+\times\mathbb{R}:y\le g^{-1}(e^x)\}=\{(x,y)\in\mathbb{R}_+\times\mathbb{R}: \log(g(y))\le x\}$ is convex, which is the case since $\log(g(y))$ is a convex function when it is non-negative ($y\ge 1$).

Because $g^{-1}(e^x)$ is concave it lies below its tangent lines. To compute the tangent we make use of
$\{(x,y)\in\mathbb{R}_+\times\mathbb{R}:y\le g^{-1}(e^x)\}=\{(x,y)\in\mathbb{R}_+\times\mathbb{R}: \log(g(y))\le x\}$. Using it we obtain that 
$g^{-1}(e^x)\le \alpha+\frac{g\left(\alpha\right)}{g'\left(\alpha\right)}\left(x-\log\left(g\left(\alpha\right)\right)\right)$ for $\alpha\ge 1$.

With this result we can go back to $g(1/z)=1+\sigma^2/\mu^2$. Observe that we have $$z=\frac{1}{g^{-1}(1+\sigma^2/\mu^2)},$$ so by our inequality,
$$z\ge\left(\alpha+\frac{g\left(\alpha\right)}{g'\left(\alpha\right)}\left(\log\left(1+\sigma^2/\mu^2\right)-\log\left(g\left(\alpha\right)\right)\right)\right)^{-1}=(c(\beta)+\beta\log (1+\sigma^2/\mu^2))^{-1},$$ where $\beta > 1$ and $c(\beta)$ a constant that depends only on $\beta$.
\end{proof}
Note that, given $\mu$ and $\sigma$ one could optimise over $\beta$ to obtain the best bound, however, the resulting problem has no closed form solution, so one might as well use the already known tight implicit lower bound.

The above result informs us that we already improve on the ratio of $\frac{1}{n}$ if $\mathrm{CV}(X_{\max})$ grows sub-exponentially, and that we recover a constant factor lower bound if we are in a typical market where $\mathrm{CV}(X_{\max})$ is uniformly bounded. It would be interesting to establish similar bounds, e.g., when marginal distributional information of the buyers is known.

\newpage
\bibliographystyle{plain}
\bibliography{references}
\newpage

\appendix
\section{Proof minimax equality}\label{proof:minimax-equality}
Before proving Theorem \ref{thm:minimax-equality} we provide several definitions and supporting results for completeness.
We begin with the definition of a tight family of measures.

\begin{definition}[Tightness of measures; p.~59 \cite{billingsley1999convergence}]
    Let $S$ be a separable metric space and $\mathcal{P}(S)$ the collection of all probability measures on $(S,\mathcal{B}(S))$. A family of probability measures $\mathcal{I}\subset\mathcal{P}(S)$ is tight if for all $\epsilon>0$ there exists a compact $K_\epsilon\subset S$ such that $\mu(K_\epsilon)>1-\epsilon$ for all $\mu\in\mathcal{I}$.
\end{definition}

We now recall several concepts related to compactness relevant for our setting.
A set is sequentially compact if every sequence has a convergent subsequence. In metrizable spaces, sequential compactness is equivalent to compactness. 
Since $\mathcal{P}(S)$ is metrizable (as $S$ is a separable metric space), compactness and sequential compactness coincide in our setting.
A set is relatively compact if its closure is compact. Lastly, we say a set is weakly compact if it is compact with respect to the weak topology. This clarifies the following two theorem statements.
\begin{lemma}[Prokhorov's theorem; Theorem 5.1 \cite{billingsley1999convergence}]
    \label{thm:prokhorov}
    Let $S$ be a separable metric space and $\mathcal{I}$ a family of probability measures on $(S,\mathcal{B}(S))$. If $\mathcal{I}$ is tight then is relatively (sequentially) compact in the weak topology.
\end{lemma}

\begin{lemma}[Von Neumann-Fan minimax theorem; Theorem 2 \cite{borwein2016very}]
    \label{thm:VonNeumannFan}
    Let $X$ be nonempty and convex, and let $Y$ be nonempty, weakly compact and convex. Let $g:X\times Y\to \mathbb{R}$ 
    be a bilinear\footnotemark function. Then, the following minimax equality holds
    $$
    \sup_{x\in X}\inf_{y\in Y}g(x,y)=\sup_{x\in X}\min_{y\in Y}g(x,y)=\min_{y\in Y}\sup_{x\in X}g(x,y)=\inf_{y\in Y}\sup_{x\in X}g(x,y).
    $$
\end{lemma}
\footnotetext{Bilinearity is not strictly necessary. Specifically, it is sufficient for $g$ to be convex with respect to $x\in X$ and concave and upper-semicontinuous with respect to $y\in Y$, and weakly continuous in $y$ when restricted to $Y$. These conditions are, however, implied by bilinearity.}
Theorem \ref{thm:minimax-equality} is now a relatively straightforward application of the above two theorems.

\begin{cleverproof}{thm:minimax-equality}
    Let $X_{\max}\sim F$. If $E[X_{\max}]=\infty$ then the result holds trivially. Now the case for $E[X_{\max}]<\infty$.
    Recall that a randomised stopping rule randomises over deterministic stopping rules, therefore a convex combination of stopping rules is a mixture. More explicitly, if $\tau_\lambda$ is a convex combination of $\tau_1$ and $\tau_2$, then
    $$
    \mathbb{P}(\tau_\lambda=i|v_1,\dots,v_i)=\lambda\mathbb{P}(\tau_1=i|v_1,\dots,v_i)+(1-\lambda)\mathbb{P}(\tau_2=i|v_1,\dots,v_i).
    $$
    Such a convex combination is still a stopping rule, so $\mathcal{T}$ is convex. 
    To show $\mathcal{I}(F,n)$ is weakly compact and convex we first give an equivalent definition,
    $$
    \mathcal{I}(F,n)=\{F_{\boldsymbol{X}}:\mathbb{R}^n_+\to[0,1]\;|\;F_{\boldsymbol{X}}\text{ a CDF and }F_{\boldsymbol{X}}(t,\dots,t)=F(t),\forall t\in\mathbb{R}_+\}.
    $$
    Let $F_{\boldsymbol{X}_1},F_{\boldsymbol{X}_2}\in\mathcal{I}(F,n)$ and $\lambda\in(0,1)$, then the convex combination $F_{\boldsymbol{X}_\lambda}$ is still a CDF and $F_{\boldsymbol{X}_\lambda}(t,\dots,t)=F(t)$, so $\mathcal{I}(F,n)$ is convex.    
    Let $K_\epsilon=[0,1+F^{-1}(1-\epsilon)]^n$, because $E[X_{\max}]<\infty$, $F^{-1}(1-\epsilon)$ is finite, so $K_\epsilon$ is compact. Let $\mu\in\mathcal{I}(F,n)$, then $\mu(K_\epsilon)=F(1+F^{-1}(1-\epsilon))>1-\epsilon$, so $\mathcal{I}(F,n)$ is tight. 
    By Lemma \ref{thm:prokhorov}, $\mathcal{I}(F,n)$ is relatively compact in the weak topology, and because $\mathcal{I}(F,n)$ is its own closure, $\mathcal{I}(F,n)$ is compact in the weak topology, that is, $\mathcal{I}(F,n)$ is weakly compact.
    
    \noindent Let $\tau\in\mathcal{T},F_{\boldsymbol{X}}\in\mathcal{I}(F,n)$ and define $$g(\tau,F_{\boldsymbol{X}})=\mathbb{E}[\boldsymbol{X}_{\tau}]=\int_{\mathbb{R}_+^n}\sum_{i=1}^nv_i\cdot\mathbb{P}(\tau=i|v_1,\dots,v_i)\mathrm{d}F_{\boldsymbol{X}}(\boldsymbol{v}).$$ By linearity of summation and integration 
    \begin{align*}
        g(\lambda\tau_1+(1-\lambda)\tau_2,F_{\boldsymbol{X}})&=\lambda g(\tau_1,F_{\boldsymbol{X}}) + (1-\lambda)g(\tau_2,F_{\boldsymbol{X}})\\
        g(\tau,\lambda F_{\boldsymbol{X}_1}+(1-\lambda) F_{\boldsymbol{X}_2})&=\lambda g(\tau,F_{\boldsymbol{X}_1}) + (1-\lambda)g(\tau,F_{\boldsymbol{X}_2}),
    \end{align*}
    so $g$ is bilinear. The result now follows Lemma \ref{thm:VonNeumannFan}.
\end{cleverproof}

\newpage
\section{Proof Theorem \ref{thm:randomised-bound-continuous-xmax} general case}\label{proof:universal-upperbound-general}
In the main text Theorem \ref{thm:randomised-bound-continuous-xmax} is proven for continuous $F_{\max}$. Here we prove it in generality. The proof is similar, but a more careful construction is required, since the intermediate value theorem cannot be applied. 

Let $I(A)=\mathbb{P}(X_{\max}\in A)\mathbb{E}[X_{\max}|X_{\max}\in A]$ and denote by $w_0$ and $w_n$ the lower and upper bound of the support of $F_{\max}$. Define $w_1=\inf\{w:I([w_0,w_1])\ge\frac{1}{n}\mathbb{E}[X_{\max}]\}$. Because $F_{\max}$ may not be continuous it can occur that $\int_{w_0}^{w_1}t\mathrm{d}F_{\max}(t)> \frac{1}{n}\mathbb{E}[X_{\max}]$, however in this case it must be that $\mathbb{P}(X_{\max}=w_1)>0$. 
To recover equality, the point mass at $w_1$ is only taken into the first partition with a probability $p_1$. Specifically, $p_1=\inf\{p\ge 0:I([w_0,w_1))+pI(\{w_1\})\ge\frac{1}{n}\mathbb{E}[X_{\max}]\}$. 
we determine $w_2$ and $p_2$ in an analogous fashion.
Let $$w_2=\inf\{w: (1-p_1)I(\{w_1\})+I((w_1,w])\ge\frac{1}{n}\mathbb{E}[X_{\max}]\},$$ and let $$p_2=\inf\{p\ge 0:(1-p_1)I(\{w_1\})+I((w_1,w_2))+pI(\{w_2\})\ge\frac{1}{n}\mathbb{E}[X_{\max}]\}.$$
This procedure can be applied recursively to obtain $w_1,\dots,w_{n-1},p_1,\dots,p_{n-1}$. Lastly, let $p_0=0$ and $p_n=1$.

From these values, random variables $X_1,\dots,X_n$ with distributions $F_1,\dots,F_n$ can be constructed that satisfy
\begin{align*}
    \xi_i=&(1-p_{i-1})\mathbb{P}(X_{\max}=w_{i-1})+\mathbb{P}(X_{\max}\in(w_{i-1},w_i))+p_{i}\mathbb{P}(X_{\max}=w_{i}),\\
    \mathbb{P}(X_i\in A)=&\frac{1-p_{i-1}}{\xi_i}\mathbb{P}(X_{\max}\in \{w_{i-1}\} \cap A)+\frac{1}{\xi_i}\mathbb{P}(X_{\max}\in(w_{i-1},w_i)\cap A)\\&+\frac{p_{i}}{\xi_i}\mathbb{P}(X_{\max}\in\{w_{i}\}\cap A).
\end{align*}
By construction these random variables have the following properties. Firstly, $\xi_i\mathbb{E}[X_i]=\xi_i \int_{\mathbb{R}}t\mathrm{d}F_i(t)=\frac{1}{n}\mathbb{E}[X_{\max}]$. 
Secondly, by telescoping, $\sum_{j=i+1}^n\xi_j=(1-p_{i})\mathbb{P}(X_{\max}=w_i)+\mathbb{P}(X_{\max}>w_i)\le \mathbb{P}(X_{\max}\ge w_i)$. 

Now, recursively define $b_i=\mathrm{argmax}_{x\in[w_{i-1}, w_i]}\{r_i(b_1,\dots,b_{i-1}, x)\}$ for $i=1,\dots,n$. By construction, $\mathbb{P}(\tau=i|b_1,\dots,b_{i-1},t)\le \mathbb{P}(\tau=i|b_1,\dots,b_i)$ for $t\in[w_{i-1}, w_i]$. With this the instance can be defined. Let $(v_1,\dots,v_n)$ be $(b_1,\dots,b_{j-1},X_j,0,\dots,0)$ with probability $\xi_j$. Let $\tau$ be an optimal stopping rule for this instance, then 
    \begin{align*}
        V^*(F_{\max}, n)&=\sum_{j=1}^n\xi_j\int_{\mathbb{R}}\sum_{i=1}^nv_i(t)\mathbb{P}(\tau=i | v_1(t),\dots,v_i(t))\mathrm{d}F_{j}(t) &(\text{Definition})\\
        &=\sum_{j=1}^n\sum_{i=1}^j\xi_j\int_{\mathbb{R}}v_i(t)\mathbb{P}(\tau=i | v_1(t),\dots,v_i(t))\mathrm{d}F_{j}(t)&(i>j \ \implies \text{zero payoff})\\
        &=\sum_{i=1}^n\xi_i\int_{\mathbb{R}}v_i(t)\mathbb{P}(\tau=i | v_1(t),\dots,v_i(t))\mathrm{d}F_{i}(t)\\
        &+\sum_{i=1}^{n-1}\sum_{j=i+1}^n\xi_j\int_{\mathbb{R}}v_i(t)\mathbb{P}(\tau=i | v_1(t),\dots,v_i(t))\mathrm{d}F_{j}(t)&(\text{Sum decomposition})\\
        &=\sum_{i=1}^n\xi_i\int_{\mathbb{R}}t\mathbb{P}(\tau=i | b_1,\dots,b_{i-1},t)\mathrm{d}F_{i}(t)\\
        &+\sum_{i=1}^{n-1}\sum_{j=i+1}^n\xi_j\int_{\mathbb{R}}b_i\mathbb{P}(\tau=i | b_1,\dots,b_i)\mathrm{d}F_{j}(t)&(\text{Definition }v_i(t))\\
        &\le\sum_{i=1}^n\xi_i\int_{\mathbb{R}}t\mathbb{P}(\tau=i | b_1,\dots,b_i)\mathrm{d}F_{i}(t)&(\text{Property }b_i\text{'s})\\
        &+\sum_{i=1}^{n-1}w_i\mathbb{P}(\tau=i | b_1,\dots,b_i)\sum_{j=i+1}^n\xi_j&(b_i\le w_i \text{ and} \int_{\mathbb{R}}\mathrm{d}F_{j}(t)=1)\\
        &\le \max_{1\le i \le n}\left\{\xi_i \int_{\mathbb{R}}t\mathrm{d}F_i(t)\right\}+\max_{1 \le i \le n}\left\{w_i\sum_{j=i+1}^n\xi_j\right\}&(\text{Convex combination})\\
        &\le \max_{1\le i \le n}\left\{\xi_i \int_{\mathbb{R}}t\mathrm{d}F_i(t)\right\}+\max_{1 \le i \le n}\left\{w_i \mathbb{P}(X_{\max}\ge w_i)\right\}&(\text{Telescoping})\\
        &\le \Pi^*(F_{\max}) + \frac{1}{n}\mathbb{E}[X_{\max}].
    \end{align*} \qed
\newpage

\section{Proofs of technical lemmas for Lemma \ref{lemma:conditional-upperbound}}\label{proofs:technical-lemmas}
\begin{cleverproof}{lemma:delta-choice}
    Because $F_{\max}$ is continuously differentiable at $p^*$ there exists a $\delta_0>0$ such that $f(\cdot;\epsilon)$ is differentiable on $(p^*-\delta_0,p^*+\delta_0)$. Let $\bar{p}(\epsilon)$ be the minimiser of $f(\cdot;\epsilon)$. $\bar{p}(\epsilon)$ converges to $p^*$ as $\epsilon\to 0$, so there is an $\epsilon_0$ such that for $\epsilon<\epsilon_0$, $\bar{p}(\epsilon)\in(p^*-\delta_0,p^*+\delta_0)$. Moreover, for $\epsilon<\epsilon_0$ $f(\cdot;\epsilon)$ is differentiable around $\bar{p}(\epsilon)$, and since $\bar{p}(\epsilon)$ is a minimiser there is a non-empty interval $(l(\epsilon),r(\epsilon))$ such that $f(\cdot;\epsilon)$ is decreasing on $(l(\epsilon),\bar{p}(\epsilon)]$ and increasing on $[\bar{p}(\epsilon),r(\epsilon))$. 
    We pick $l$ and $r$ such that $(l(\epsilon),r(\epsilon))$ is the largest such interval contained in $(p^*-\delta_0,p^*+\delta_0)$. 
    $l$ and $r$ are continuous in $\epsilon$ and $(l(0),r(0))$ is non-empty, so there exists an $\epsilon_0'$ such that $\bigcap\limits_{\epsilon<\epsilon_0'}(l(\epsilon),r(\epsilon))$ is non-empty and has $p^*$ in its interior.
    Therefore there is a $\delta>0$ such that $(p^*-\delta,p^*+\delta)\subset \bigcap\limits_{\epsilon<\epsilon_0'}(l(\epsilon),r(\epsilon))$. 
    
    For such a $\delta$ and $\epsilon<\epsilon_0'$, $M$ is increasing ($\mathbb{P}(X_{\max}\ge t)$ is decreasing)), continuous on $[p^*-\delta,p^*+\delta]$ (Because it is differentiable), $f$ is non-negative, decreasing on $(p^*-\delta,\bar{p}(\epsilon)]$ and increasing on $[\bar{p}(\epsilon),p^*+\delta)$.
\end{cleverproof}
\begin{remark}
    The essence of Lemma \ref{lemma:delta-choice} is encapsulated in Figure \ref{fig:delta-essence}.
    \begin{figure}[ht]
    \centering
    \begin{tikzpicture}[xscale=1.5, yscale=0.8]
        \draw[->, thick] (0,0) -- (6,0) node[below right] {$\epsilon$};
        \draw[thick] (0,-3.2) -- (0,3.2);

        \foreach \y/\label in {
            2.5/{$p^* + \delta_0$},
            1.2/{$p^* + \delta$},
            0/{$p^*$},
            -1.2/{$p^* - \delta$},
            -2.5/{$p^* - \delta_0$}
        } {
            \draw[thick] (0,\y) -- (-0.1,\y) node[left=3pt] {\label};
        }

        \draw[dashed, thick] (0,2.5) -- (6,2.5);
        \draw[dashed, thick] (0,-2.5) -- (6,-2.5);

        \draw[dashed, thick] (0,-1.2) -- (3,-1.2) -- (3,1.2) -- (0,1.2);

        \draw[->, thick] (3.2,-0.23) -- (3.05, -0.09);
        \node at (3.35,-0.3) {$\epsilon_0'$};

        \draw[smooth, thick] plot[domain=0:6] (\x, {1.9 - 0.5*sin(deg(\x))}) node[right] {$r(\epsilon)$};;
        \draw[smooth, thick] plot[domain=0:6] (\x, {-2 + 0.5*cos(deg(\x)) + 0.4*\x}) node[right] {$l(\epsilon)$};;

    \end{tikzpicture}
    \caption{The essence of Lemma \ref{lemma:delta-choice}}
    \label{fig:delta-essence}
\end{figure}
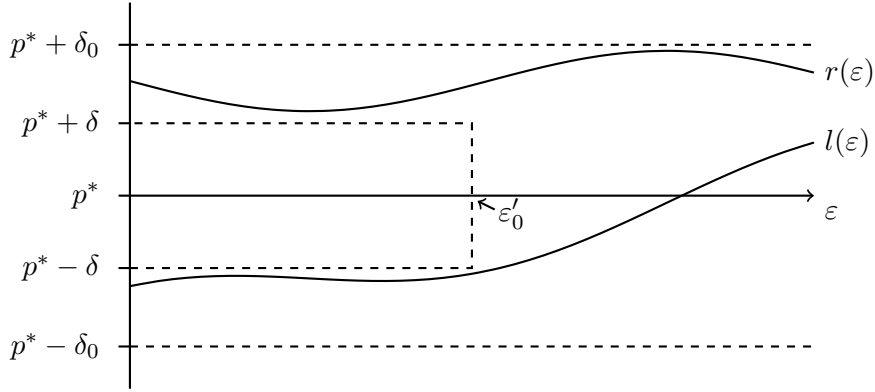
    What the lemma shows is that the dashed rectangle around $p^*$ between $l(\epsilon)$ and $r(\epsilon)$ exists.
\end{remark}
\begin{lemma}
    \label{lemma:stopping-time-difference-equation-1}
    Let $I$ be a continuous and increasing function which can be written as $I(t)=t+f(t)$ where $f$ is a decreasing and non-negative function on $[a,b]$ and let $I^k=I\circ I^{k-1}$, then
    $$
    \inf\{k\in\mathbb{N} : I^k(a)\ge b\}\le 1 + \int_a^bf(t)^{-1}\mathrm{d}t.
    $$
\end{lemma}
\begin{proof}
    Let $\tilde{I}(t;\delta)=t+\delta f(t)$ and $\tilde{I}^k=\tilde{I}\circ \tilde{I}^{k-1}$. For $t$ such that $t\ge a$ and $t+\delta f(t)\le b$, we have $t+\delta f(t)\ge t\ge a$ by the non-negativity of $f$, and $f(t)\ge f(t+\delta f(t))$ since $f$ is decreasing on $[a,b]$ and $t+\delta f(t)\in[a,b]$. Therefore,
    \begin{align*}
        \tilde{I}^2(t;\delta)&=\tilde{I}(\tilde{I}(t;\delta);\delta)=\tilde{I}(t;\delta)+\delta f(\tilde{I}(t;\delta))=t+\delta f(t)+\delta f(t+\delta f(t))\\
        &\le t+2\delta f(t)=\tilde{I}(t;2\delta).
    \end{align*}    
    Iterative application yields that $\tilde{I}^m(t;\delta)\le \tilde{I}(t;m\delta)$. Put differently, for $\delta^{-1}\in\mathbb{N}$, $I(t)=\tilde{I}(t;1)\ge\tilde{I}^{1/\delta}(t;\delta)$. We now show $I^k(t)\ge \tilde{I}^{k/\delta}(t;\delta)$ by induction. It is true for $k=1$ and $I^{k+1}(t)=I^k(I(t))\ge I^k(\tilde{I}^{1/\delta}(t;\delta))$, because $I^k$ is increasing, now $I^k(\tilde{I}^{1/\delta}(t;\delta))\ge \tilde{I}^{k/\delta}(\tilde{I}^{1/\delta}(t;\delta);\delta)=\tilde{I}^{(k+1)/\delta}(t;\delta)$.
    
    Let $x_\delta(s;a)=\tilde{I}^{s/\delta}(a;\delta)$ with $s$ an integer multiple of $\delta$, then
    $x_\delta(s+\delta;a)=\tilde{I}^{1+s/\delta}(a;\delta)=\tilde{I}(x_\delta(s;a);\delta)$, therefore $\frac{1}{\delta}(x_\delta(s+\delta;a)-x_\delta(s;a))=f(x_\delta(s;a))$. Let $x(s;a)=\lim_{\delta\to 0}x_\delta(s;a)$, where the limit is over $\delta$ such that $s/\delta$ is an integer. By continuity of $I$, $x$ satisfies $x'(s;a)=f(x(s;a))$ and $x(0;a)=a$, moreover $x(k;a)\le I^k(a)$ for $k\in\mathbb{N}$. 
    Lastly, observe that $\frac{\mathrm{d}}{\mathrm{d}s}x^{-1}(s;a)=f(s)^{-1}$ and that $x^{-1}(a;a)=0$. The result now follows,
    \begin{align*}
        &\inf\{k\in\mathbb{N} : I^k(a)\ge b\}\le \inf\{k\in\mathbb{N} : x(k;a)\ge b\}\\
        \le&1+\inf\{t\in\mathbb{R} : x(t;a)\ge b\}= 1 + x^{-1}(b;a)= 1 + \int_a^bf(t)^{-1}\mathrm{d}t. \qedhere
    \end{align*}
\end{proof}
\begin{remark}
    The intuition behind Lemma \ref{lemma:stopping-time-difference-equation-1} is that if $I$ satisfies the conditions, then more frequent updates of the difference equation $w_{i+1}=I(w_i)$, that is a smaller $\delta$, monotonically decreases the trajectory. 
    In Figure \ref{fig:I-tilde-process} this is illustrated for a specific example, namely $I(t)=\frac{9}{10(1+t)e^{-t}}$, $a=0$ and $b=3/2$.
    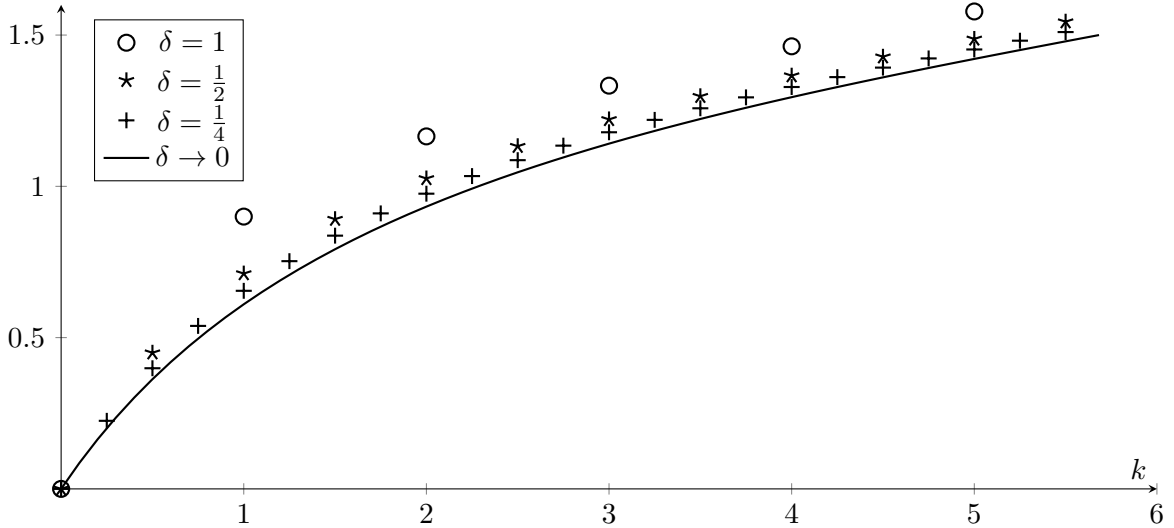
\begin{figure}[!h]
        \centering
        \begin{tikzpicture}
          \begin{axis}[
            axis lines=middle,
            xlabel={$k$},
            xmin=0, ymin=0,
            xmax=6, ymax=1.6,
            xtick distance=1,
            ytick distance=0.5,
            grid=none,
            legend pos=north west,
            width=16cm, height=8cm
          ]
        
          \addplot[
            only marks,
            mark=o,
            mark size=3pt,
            color=black,
            thick
          ] coordinates {
            (0, 0) (1, 0.9) (2, 1.165075) (3, 1.332770) (4, 1.462801) (5, 1.577973)
          };
          \addlegendentry{$\delta=1$}

          \addplot[
            only marks,
            mark=star,
            mark size=3pt,
            color=black,
            thick
          ] coordinates {
            (0,0) (0.5,0.45) (1,0.7117) (1.5,0.8915) (2,1.026) (2.5,1.1326) (3,1.2212) (3.5,1.2977) (4,1.3658) (4.5,1.4283) (5,1.4872) (5.5,1.5442)
          };
          \addlegendentry{$\delta=\frac{1}{2}$}
        
          \addplot[
            only marks,
            mark=+,
            mark size=3pt,
            color=black,
            thick
          ] coordinates {
            (0,0) (0.25,0.225) (0.5,0.3988) (0.75,0.5387) (1,0.6547) (1.25,0.7527) (1.5,0.837) (1.75,0.9106) (2,0.9757) (2.25,1.0339) (2.5,1.0865) (2.75,1.1345) (3,1.1787) (3.25,1.2197) (3.5,1.258) (3.75,1.2941) (4,1.3283) (4.25,1.361) (4.5,1.3924) (4.75,1.4228) (5,1.4524) (5.25,1.4814) (5.5,1.5099)
          };
          \addlegendentry{$\delta=\frac{1}{4}$}

          \addplot[
            color=black,
            thick
          ] coordinates {
            (0,0) (0.1,0.0858) (0.2,0.164) (0.3,0.2356) (0.4,0.3015) (0.5,0.3625) (0.6,0.4191) (0.7,0.4718) (0.8,0.5209) (0.9,0.567) (1,0.6102) (1.1,0.6508) (1.2,0.6891) (1.3,0.7253) (1.4,0.7596) (1.5,0.7921) (1.6,0.823) (1.7,0.8525) (1.8,0.8805) (1.9,0.9073) (2,0.9329) (2.1,0.9575) (2.2,0.981) (2.3,1.0036) (2.4,1.0254) (2.5,1.0464) (2.6,1.0666) (2.7,1.0861) (2.8,1.105) (2.9,1.1233) (3,1.1411) (3.1,1.1583) (3.2,1.175) (3.3,1.1913) (3.4,1.2071) (3.5,1.2226) (3.6,1.2376) (3.7,1.2524) (3.8,1.2668) (3.9,1.2809) (4,1.2947) (4.1,1.3083) (4.2,1.3216) (4.3,1.3346) (4.4,1.3475) (4.5,1.3602) (4.6,1.3727) (4.7,1.385) (4.8,1.3972) (4.9,1.4092) (5,1.4211) (5.1,1.4329) (5.2,1.4446) (5.3,1.4562) (5.4,1.4677) (5.5,1.4792) (5.6,1.4906) (5.683,1.5) 
          };
          \addlegendentry{{$\delta\to0$}}
        
          \end{axis}
        \end{tikzpicture}
        \caption{The processes $\{\tilde{I}^{k/\delta}(a;\delta)\}_{k=0}^\infty$}
        \label{fig:I-tilde-process}
    \end{figure}
\end{remark}

\begin{lemma}
    \label{lemma:stopping-time-difference-equation-2}
    Let $I$ be a continuous and increasing function which can be written as $I(t)=t+f(t)$ where $f$ is an increasing and non-negative function on $[a,b]$. Let $I^k=I\circ I^{k-1}$ and $\tilde{f}(t)=t-I^{-1}(t)$, then
    $$
    \inf\{k\in\mathbb{N}  : I^k(a)\ge b\}\le 1 + \int_a^b\tilde{f}(t)^{-1}\mathrm{d}t.
    $$
\end{lemma}
\begin{proof}
    Let $I^{-k}=I^{-1}\circ I^{-(k-1)}$. Observe that $\inf\{k : I^k(a)\ge b\}=\inf\{k : I^{-k}(b)\le a\}$. Because $f$ is increasing, the sequence $\{I^{-k}(b)\}$ decreases with decreasing increments, therefore $I^{-1}(t)=t-\tilde{f}(t)$ with $\tilde{f}$ an increasing function. The remainder of the proof is analogous to Lemma \ref{lemma:stopping-time-difference-equation-1}, so the details are omitted. Let $\tilde{I}(t;\delta)=t-\delta \tilde{f}(t)$ and $\tilde{I}^k=\tilde{I}\circ \tilde{I}^{k-1}$ then,
    \begin{align*}
        \tilde{I}^2(t;\delta)&=\tilde{I}(\tilde{I}(t;\delta);\delta)=\tilde{I}(t;\delta)-\delta f(\tilde{I}(t;\delta))=t-\delta f(t)-\delta f(t-\delta f(t))\\
        &\ge t-2\delta f(t)=\tilde{I}(t;2\delta).
    \end{align*}
    From this inequality we obtain $I^{-k}(t)=\tilde{I}^k(t;1)\le x(k;t)$ where $x$ satisfies $x'(s;t)=\tilde{f}(x(s;t))$ and $x(0;t)=t$. Consequently,
    \begin{align*}
        &\inf\{k\in\mathbb{N} : I^k(a)\ge b\}=\inf\{k\in\mathbb{N}  : I^{-k}(b)\le a\}\le \inf\{k\in\mathbb{N} : x(k;b)\le a\}\\
        \le&1+\inf\{t\in\mathbb{R} : x(t;b)\le a\}= 1 + x^{-1}(a;b)= 1 - \int_b^a\tilde{f}(t)^{-1}\mathrm{d}t=1 + \int_a^b\tilde{f}(t)^{-1}\mathrm{d}t. \quad \qedhere
    \end{align*}
\end{proof}

\begin{lemma}\label{lemma:technical-upperbound-argument}
Let $\Pi^*=\Pi^*(F_{\max})$, $M(t;\epsilon):=\frac{\Pi^*+\epsilon}{\mathbb{P}(X_{\max}\ge t)}$, $M^k=M\circ M^{k-1}$, $k(a,b;\epsilon)=\inf\{k:M^k(a;\epsilon)\ge b\}$ and $\delta$ chosen according to Lemma \ref{lemma:delta-choice}, then 
    $$
    \lim_{\epsilon\to 0}\epsilon^{\frac{1}{2}}k(p^*-\delta,p^*+\delta;\epsilon)\le\sqrt{2\pi^2\frac{(1-F_{\max}(p^*))^2}{p^*F''_{\max}(p^*)+2F'_{\max}(p^*)}}.
    $$
\end{lemma}
\begin{proof}
    Let $f(t;\epsilon)=M(t;\epsilon)-t$ and $\bar{p}$ be the minimiser of $f$. By the triangle like inequality,
    $$
    k(p^*-\delta,p^*+\delta;\epsilon)\le k(p^*-\delta,\bar{p};\epsilon)+k(\bar{p},p^*+\delta;\epsilon).
    $$
    Choosing $\delta$ according to Lemma \ref{lemma:delta-choice} allows us to apply Lemmas \ref{lemma:stopping-time-difference-equation-1} and \ref{lemma:stopping-time-difference-equation-2}, which give us the following upper bounds on the above two terms
    \begin{align}
        k(p^*-\delta,\bar{p};\epsilon)&\le 1+\int_{p^*-\delta}^{\bar{p}}f(t;\epsilon)^{-1}\mathrm{d}t,\label{eq:bound-left}\\
        k(\bar{p},p^*+\delta;\epsilon)&\le1+\int_{\bar{p}}^{p^*+\delta}\left(t-M^{-1}(t;\epsilon)\right)^{-1}\mathrm{d}t.\label{eq:bound-right}
    \end{align}
    $f(t;\epsilon)$ is known and $t-M^{-1}(t;\epsilon)=t-F^{-1}_{\max}\left(1-\frac{\Pi^*+\epsilon}{t}\right)$.
    We now proceed with the limit argument.    
    Let $\zeta_1(t;\epsilon)=f(t;\epsilon)$ and $\zeta_2(t;\epsilon)=t-M^{-1}(t;\epsilon)$. For $|s|\le \delta$ we have by Taylor's theorem that,
    $$
    \zeta_i(p^*+s)\ge\xi_i(p^*+s;\epsilon):=\zeta_i(p^*)+\zeta_i'(p^*)s+\frac{1}{2}\zeta_i''(p^*)s^2-L|s|^3, 
    $$
    where $L=\max\{L_1,L_2\}$ and $L_i=\sup\{|\frac{1}{6}\zeta_i'''(p^*+s)|:|s|\le\delta\}$. The $\xi_i$'s can be written as
    $$
    \xi_i(p^*+s;\epsilon):=a_i+b_is+c_is^2-L|s|^3,
    $$
    where
    \begin{align*}    
        a_1&=\frac{\Pi^*+\epsilon}{1-F_{\max}(p^*)}-p^*=\frac{\epsilon}{1-F_{\max}(p^*)},\\
        a_2&=p^*-F^{-1}_{\max}\left(1-\frac{\Pi^*+\epsilon}{p^*}\right),\\
        b_1&=\frac{(\Pi^*+\epsilon)F'_{\max}(p^*)}{(1-F_{\max}(p^*))^2}-1,\\
        b_2&=1-\frac{\Pi^*+\epsilon}{(p^*)^2F'_{\max}\left(F^{-1}_{\max}\left(1-\frac{\Pi^*+\epsilon}{p^*}\right)\right)},\\
        c_1&=(\Pi^*+\epsilon)\frac{F''_{\max}(p^*)(1-F_{\max}(p^*))+2F'_{\max}(p^*)^2}{2(1-F_{\max}(p^*))^3},\\
        c_2&=\frac{(\Pi^*+\epsilon)}{(p^*)^3F'_{\max}\left(F^{-1}_{\max}\left(1-\frac{\Pi^*+\epsilon}{p^*}\right)\right)}+\frac{(\Pi^*+\epsilon)^2F''_{\max}\left(F^{-1}_{\max}\left(1-\frac{\Pi^*+\epsilon}{p^*}\right)\right)}{2(p^*)^4F'_{\max}\left(F^{-1}_{\max}\left(1-\frac{\Pi^*+\epsilon}{p^*}\right)\right)^3}.
    \end{align*}
    We will now show that $a_i=a\epsilon+O(\epsilon^2)$, $b_i=O(\epsilon)$ and $c_i=c+O(\epsilon)$ where $a=(1-F_{\max}(p^*))^{-1}$ and $c=\frac{p^*F''_{\max}(p^*)+2F'_{\max}(p^*)}{2(1-F_{\max}(p^*))}$. 
    
    Recall that $1-F_{\max}(p^*)=p^*F'_{\max}(p^*)$, since $p^*$ solves the first order condition.
    Firstly, $a_1=a\epsilon$ and $a_2=a\epsilon+O(\epsilon^2)$ follows from the fact that $a_2=0$ for $\epsilon=0$ and that 
    \begin{align*}
        \left. \frac{\mathrm{d} a_2}{\mathrm{d} \epsilon} \right|_{\epsilon = 0}&=
    \left. \left(p^*F'_{\max}\left(F^{-1}_{\max}\left(1-\frac{\Pi^*+\epsilon}{p^*}\right)\right)\right)^{-1} \right|_{\epsilon = 0}\\
    &=(p^*F'_{\max}(p^*))^{-1}=(1-F_{\max}(p^*))^{-1}.
    \end{align*}
    Secondly, for $\epsilon=0$ 
    \begin{align*}    
        b_1&=\frac{(\Pi^*)F'_{\max}(p^*)}{(1-F_{\max}(p^*))^2}-1=\frac{p^*F'_{\max}\left(p^*\right)}{1-F_{\max}(p^*)}-1=0,\\
        b_2&=1-\frac{\Pi^*}{(p^*)^2F'_{\max}\left(F^{-1}_{\max}\left(1-\frac{\Pi^*}{p^*}\right)\right)}=1-\frac{1-F_{\max}(p^*)}{p^*F'_{\max}\left(p^*\right)}=0.
    \end{align*}
    Therefore $b_1$ and $b_2$ are $O(\epsilon)$. Lastly, for $\epsilon=0$ 
    \begin{align*}    
        c_1&=(\Pi^*)\frac{F''_{\max}(p^*)(1-F_{\max}(p^*))+2F'_{\max}(p^*)^2}{2(1-F_{\max}(p^*))^3}\\
        &=\frac{p^*F''_{\max}(p^*)(1-F_{\max}(p^*))+2p^*F'_{\max}(p^*)^2}{2(1-F_{\max}(p^*))^2}\\
        &=\frac{p^*F''_{\max}(p^*)+2F'_{\max}(p^*)}{2(1-F_{\max}(p^*))},\\
        c_2&=\frac{\Pi^*}{(p^*)^3F'_{\max}\left(F^{-1}_{\max}\left(1-\frac{\Pi^*}{p^*}\right)\right)}+\frac{(\Pi^*)^2F''_{\max}\left(F^{-1}_{\max}\left(1-\frac{\Pi^*}{p^*}\right)\right)}{2(p^*)^4F'_{\max}\left(F^{-1}_{\max}\left(1-\frac{\Pi^*}{p^*}\right)\right)^3}\\
        &=\frac{1-F_{\max}(p^*)}{(p^*)^2F'_{\max}\left(p^*\right)}+\frac{(1-F_{\max}(p^*))^2F''_{\max}\left(p^*\right)}{2(p^*)^2F'_{\max}\left(p^*\right)^3}\\
        &=\frac{1}{p^*}+\frac{F''_{\max}\left(p^*\right)}{2F'_{\max}\left(p^*\right)}=\frac{p^*F''_{\max}(p^*)+2F'_{\max}(p^*)}{2p^*F'_{\max}(p^*)}=\frac{p^*F''_{\max}(p^*)+2F'_{\max}(p^*)}{2(1-F_{\max}(p^*))}.
    \end{align*}
    Therefore, $c_i=c+O(\epsilon)$. From (\ref{eq:bound-left}) and (\ref{eq:bound-right}) and the fact that $\zeta_i$ is lower bounded by $\xi_i$ we obtain that
    \begin{align*}
        \lim_{\epsilon\to 0}\epsilon^{\frac{1}{2}}\left(k(p^*-\delta,\bar{p};\epsilon)+k(\bar{p},p^*+\delta;\epsilon)\right)&\le\lim_{\epsilon\to 0}\epsilon^{\frac{1}{2}}\left(\int_{p^*-\delta}^{\bar{p}}\frac{\mathrm{d}t}{\zeta_1(t;\epsilon)}+\int_{\bar{p}}^{p^*+\delta}\frac{\mathrm{d}t}{\zeta_2(t;\epsilon)}\right)\\
        &\le\lim_{\epsilon\to 0}\epsilon^{\frac{1}{2}}\left(\int_{p^*-\delta}^{\bar{p}}\frac{\mathrm{d}t}{\xi_1(t;\epsilon)}+\int^{p^*+\delta}_{\bar{p}}\frac{\mathrm{d}t}{\xi_2(t;\epsilon)}\right).
    \end{align*}
    The value of the right-hand side is driven by the first order dependence of $a_i,b_i$ and $c_i$ on $\epsilon$. Since these are the same for $i=1$ and $i=2$, we have that for $\xi_0(p^*+s;\epsilon):=a_0+b_0s+c_0s^2-L|s|^3$ with $a_0=a\epsilon+O(\epsilon^2)$, $b_0=O(\epsilon)$ and $c_0=c+O(\epsilon)$ that 
    \begin{align}
        \lim_{\epsilon\to 0}\epsilon^{\frac{1}{2}}\left(\int_{p^*-\delta}^{\bar{p}}\frac{\mathrm{d}t}{\xi_1(t;\epsilon)}+\int^{p^*+\delta}_{\bar{p}}\frac{\mathrm{d}t}{\xi_2(t;\epsilon)}\right)=\lim_{\epsilon\to 0}\epsilon^{\frac{1}{2}}\int_{p^*-\delta}^{p^*+\delta}\frac{\mathrm{d}t}{\xi_0(t;\epsilon)}.\label{eq:xi-0}
    \end{align}
    The function $\xi_0$ has roots above and below $p^*$, so $\delta$ needs to be small enough such that there are no roots in the interval $(p^*-\delta,p^*+\delta)$. To do so we make use of two facts. Firstly, as $\epsilon$ decreases so does $\xi_0$, resulting in roots closer to $p^*$. For $\epsilon=0$, $\xi_0$ has roots at $p^*\pm\frac{c}{L}$, so $\delta <\frac{c}{L}$ is necessary. Moreover, for $\epsilon=0$, $\xi_0$ has local maxima at $p^*\pm\frac{2}{3}\frac{c}{L}$, so we pick $\delta\le\frac{2}{3}\frac{c}{L}$. 
    Such a $\delta$ exists, since $L$ is continuous in $\delta$ at zero, because $F_{\max}$ is three times continuously differentiable at $p^*$. This is not in conflict with Lemma \ref{lemma:delta-choice}, since $\delta$ can always be chosen smaller.
    
    Denote the unique real root of $a_0+b_0t+c_0t^2+Lt^3=0$ by $r_1$ and the unique real root of $a_0+b_0t+c_0t^2-Lt^3=0$ by $r_2$. We have that $r_1=-\frac{c}{L}+O(\epsilon)$ and $r_2=\frac{c}{L}+O(\epsilon)$.   
    Partial fraction decomposition yields that
    $$
    \frac{1}{\xi_0(p^*+s;\epsilon)}=\begin{cases}
        \lambda_1\left(\frac{1}{s-r_1}-\frac{Ls+c_0+2r_1L}{Ls^{2}+\left(c_0+r_1L\right)s-a_0/r_1}\right),&\text{if } s\le 0,\\
        \lambda_2\left(\frac{1}{s-r_{2}}-\frac{Ls+2r_{2}L-c_0}{Ls^{2}+\left(r_{2}L-c_0\right)s+a_0/r_{2}}\right),&\text{if } s> 0,
    \end{cases}
    $$
    where $\lambda_1=\frac{r_1}{r_1^{2}\left(c_0+2r_1L\right)-a_0}$ and $\lambda_2=\frac{r_{2}}{r_{2}^{2}\left(c_0-2r_{2}L\right)-a_0}$. (\ref{eq:xi-0}) can now be evaluated,
    \begin{align*}
        \int_{p^*-\delta}^{p^*+\delta}\frac{\mathrm{d}s}{\xi_0(s;\epsilon)}=&\lambda_1\int_{-\delta}^0\left(\frac{1}{s-r_1}-\frac{Ls+c_0+2r_1L}{Ls^{2}+\left(c_0+r_1L\right)s-a_0/r_1}\right)\mathrm{d}s\\
        +&\lambda_2\int_0^{\delta}\left(\frac{1}{s-r_{2}}-\frac{Ls+2r_{2}L-c_0}{Ls^{2}+\left(r_{2}L-c_0\right)s+a_0/r_{2}}\right)\mathrm{d}s\\
        =&\lambda_1\log\left|r_{1}\right|-\lambda_1\log\left|r_{1}+\delta\right|-\lambda_1\frac{1}{2}\log\left|\frac{a_{0}}{r_{1}}\right|\\-&\lambda_1\frac{2(c_{0}+2r_{1}L)-\left(c_{0}+r_{1}L\right)}{\sqrt{-4L\frac{a_{0}}{r_{1}}-\left(c_{0}+r_{1}L\right)^{2}}}\arctan\left(\frac{\left(c_{0}+r_{1}L\right)}{\sqrt{-4L\frac{a_{0}}{r_{1}}-\left(c_{0}+r_{1}L\right)^{2}}}\right)\\+&\lambda_1\frac{1}{2}\log\left|L\delta^{2}-\left(c_{0}+r_{1}L\right)\delta-\frac{a_{0}}{r_{1}}\right|\\+&\lambda_1\frac{2(c_{0}+2r_{1}L)-\left(c_{0}+r_{1}L\right)}{\sqrt{-4L\frac{a_{0}}{r_{1}}-\left(c_{0}+r_{1}L\right)^{2}}}\arctan\left(\frac{-2L\delta+\left(c_{0}+r_{1}L\right)}{\sqrt{-4L\frac{a_{0}}{r_{1}}-\left(c_{0}+r_{1}L\right)^{2}}}\right)\\
        +&\lambda_2\left(\log\left|\delta-r_{2}\right|-\log\left|r_{2}\right|+\frac{1}{2}\log\left|\frac{a_{0}}{r_{2}}\right|\right)\\+&\lambda_2\frac{2\left(2r_{2}L-c_{0}\right)-\left(r_{2}L-c_{0}\right)}{\sqrt{4L\frac{a_{0}}{r_{2}}-\left(r_{2}L-c_{0}\right)^{2}}}\arctan\left(\frac{\left(r_{2}L-c_{0}\right)}{\sqrt{4L\frac{a_{0}}{r_{2}}-\left(r_{2}L-c_{0}\right)^{2}}}\right)\\-&\lambda_2\frac{1}{2}\log\left|L\delta^{2}+\left(r_{2}L-c_{0}\right)\delta+\frac{a_{0}}{r_{2}}\right|\\-&\lambda_2\frac{2\left(2r_{2}L-c_{0}\right)-\left(r_{2}L-c_{0}\right)}{\sqrt{4L\frac{a_{0}}{r_{2}}-\left(r_{2}L-c_{0}\right)^{2}}}\arctan\left(\frac{2L\delta+\left(r_{2}L-c_{0}\right)}{\sqrt{4L\frac{a_{0}}{r_{2}}-\left(r_{2}L-c_{0}\right)^{2}}}\right)
    \end{align*}
    Every term except the fourth and eighth vanishes when we multiply with $\epsilon^{\frac{1}{2}}$ and take the limit. The arctans converge to $\pm\pi$, as their arguments diverge, leaving only simple terms. Rudimentary calculus reveals that $\epsilon^{\frac{1}{2}}\int_{p^*-\delta}^{p^*+\delta}\frac{\mathrm{d}s}{\xi_0(s;\epsilon)}$ converges to $\pi(ac)^{-\frac{1}{2}}$. Substituting the expressions of $a$ and $c$ yields the result.
\end{proof}
\newpage

\end{document}